\documentclass[reqno,dvipdfmx]{amsart}
\usepackage[margin=20truemm,tmargin=20truemm,lmargin=25truemm,rmargin=25truemm]{geometry}

\usepackage{amssymb}
\usepackage{mathrsfs}
\usepackage{mathtools}
\usepackage{aligned-overset}

\mathtoolsset{showonlyrefs, showmanualtags}

\usepackage{bm}
\usepackage{comment}
\usepackage{amsthm}
\usepackage{mleftright}
\usepackage{enumerate}
\usepackage{enumitem}
\usepackage{refcount}
\usepackage{xparse}
\newenvironment{eqenumerate}
{\begin{enumerate}[ref=\thesection.\theenumi]
		
		\setcounter{enumi}{\value{equation}}}
	{\setcounter{equation}{\value{enumi}}
\end{enumerate}}
\NewDocumentCommand\eqitem{ o }{
	\setcounter{enumi}{\value{equation}}
	\IfValueTF{#1}
	{\item[#1]}
	{\item}
	\setcounter{equation}{\value{enumi}}
}

\newtheorem{theorem}{Theorem}[section]
\newtheorem{corollary}[theorem]{Corollary}
\newtheorem{lemma}[theorem]{Lemma}

\newtheorem{assumption}[theorem]{Assumption}

\newtheorem*{Acknow*}{Acknowledgement}
\numberwithin{equation}{section}

\newtheorem{quotetheorem}[theorem]{Theorem}
\newtheorem{quotelemma}[theorem]{Lemma}

\theoremstyle{remark}

\newtheorem*{remark}{Remark}

\renewcommand{\S}{\mathbb{S}}
\DeclareMathOperator{\supp}{supp}

\newcommand{\R}{\mathbb{R}}
\newcommand{\N}{\mathbb{N}}
\newcommand{\C}{\mathbb{C}}
\newcommand{\Z}{\mathbb{Z}}
\renewcommand{\S}{\mathbb{S}}
\newcommand{\positiveR}{\R_{>0}}

\newcommand{\conjugate}{\overline}

\newcommand*{\variabledot}{\makebox[1ex]{\textbf{$\cdot$}}}
\newcommand{\adjoint}[1]{#1^{*}}

\newcommand\restr[2]{{
		\left.\kern-\nulldelimiterspace 
		#1 
		\vphantom{\big|} 
		\right|_{#2} 
}}

\newcommand{\sumnk}
{\sum_{n=-k-1}^{k}}

\newcommand{\rad}{\textup{rad.}}

\newcommand{\DiracS}{\widetilde{S}}
\newcommand{\DiracL}{\widetilde{L}}

\newcommand{\abs}[1]{\lvert#1\rvert} 

\newcommand{\norm}[1]{\lVert#1\rVert} 
 
\newcommand{\set}[2]{\{ \, #1 : #2 \, \} } 
 
\newcommand{\measure}[1]{ \lvert#1\rvert }
\newcommand{\innerproduct}[2]{\langle #1, #2 \rangle}
\newcommand{\Innerproduct}[2]{\mleft \langle #1, #2 \mright \rangle}

\newcommand{\HP}{ \mathcal{H} }
\newcommand{\Pkn}[2]{P_{#1}^{#2}}

\newcommand{\GPpn}[2]{C^{#1}_{#2}}

\newcommand{\matrixCkn}[2]{\mathcal{C}_{#1}^{#2}}

\newcommand{\Normalizekn}[2]{N_{#1}^{#2}}

\newcommand{\transpose}[1]{\prescript{\top\!}{}{#1}}

\newcommand{\matrixO}{O}
\newcommand{\matrixI}{I}

\newcommand{\Akn}[2]{A_{#1}^{#2}}
\newcommand{\transposeAkn}[2]{\transpose{\Akn{#1}{#2}}}

\newcommand{\Lambdak}[1]{\Lambda_{#1}}
\newcommand{\DiracLambdak}[1]{\widetilde{\Lambda}_{#1}}

\DeclareMathAlphabet{\mymathbb}{U}{BOONDOX-ds}{m}{n}
\newcommand{\zerovec}{0}
\newcommand{\ukn}[2]{u_{#1}^{#2}}
\newcommand{\vkn}[2]{v_{#1}^{#2}}

\newcommand{\fpm}{ f_{\pm} }

\newcommand{\matrixEk}[1]{{E}_{#1}}

\newcommand{\fk}[1]{f_{#1}}
\newcommand{\fkpm}[1]{ f_{#1, \pm} }
\newcommand{\fkp}[1]{ f_{#1, +} }
\newcommand{\fkm}[1]{ f_{#1, -} }

\newcommand{\Ykn}[2]{Y_{#1}^{#2}}

\newcommand{\matrixEkn}[2]{{E}_{#1}^{#2}}
\newcommand{\matrixYkn}[2]{\mathcal{Y}_{#1}^{#2}}

\newcommand{\fkn}[2]{f_{#1}^{#2}}

\newcommand{\fknpm}[2]{f_{#1, \pm}^{#2}}
\newcommand{\fknp}[2]{f_{#1, +}^{#2}}
\newcommand{\fknm}[2]{f_{#1, -}^{#2}}

\newcommand{\Diraclambda}{\widetilde{\lambda}}
\newcommand{\Diraclambdak}[1]{\Diraclambda_{#1}}
\newcommand{\Diraclambdasup}{\Diraclambda^*}

\DeclareFontFamily{U}{matha}{\hyphenchar\font45}
\DeclareFontShape{U}{matha}{m}{n}{
	<5> <6> <7> <8> <9> <10> gen * matha
	<10.95> matha10 <12> <14.4> <17.28> <20.74> <24.88> matha12
}{}
\DeclareSymbolFont{matha}{U}{matha}{m}{n}
\DeclareFontFamily{U}{mathx}{\hyphenchar\font45}
\DeclareFontShape{U}{mathx}{m}{n}{
	<5> <6> <7> <8> <9> <10>
	<10.95> <12> <14.4> <17.28> <20.74> <24.88>
	mathx10
}{}
\DeclareSymbolFont{mathx}{U}{mathx}{m}{n}

\DeclareMathSymbol{\obot}         {2}{matha}{"6B}
\DeclareMathSymbol{\bigobot}       {1}{mathx}{"CB}

\newcommand{\Hermiteop}{{H}}
\newcommand{\Hilbert}{\mathcal{H}}

\newcommand{\spectral}[1]{E_{#1}}

\usepackage[numbers,sort]{natbib}
\providecommand*{\doi}[1]{doi:\href{https://doi.org/#1}{#1}}
\renewcommand*{\MR}[1]{\href{https://mathscinet.ams.org/mathscinet-getitem?mr=#1}{#1}}
\usepackage{hyperref}
\hypersetup{
	setpagesize=false,
	bookmarksnumbered=true,
	bookmarksopen=true,
	colorlinks=true,
	linkcolor=blue,
	citecolor=blue,
}

\title[Smoothing estimates for the 3D Dirac equation]{Optimal constants of smoothing estimates for the 3D Dirac equation}
\date{}

\author{Makoto Ikoma}
\address[Makoto Ikoma]{Graduate School of Mathematics, Nagoya University, Furocho, Chikusaku, Nagoya, Aichi, 464-8602, Japan}
\email{ikma.m18005d@gmail.com}

\author{Soichiro Suzuki}
\address[Soichiro Suzuki]{Department of Mathematics, Chuo University, 1-13-27, Kasuga, Bunkyo-ku, Tokyo, 112-8551, Japan}
\email{soichiro.suzuki.m18020a@gmail.com}
\thanks{The second author was supported by Japan Society for the Promotion of Science (JSPS) KAKENHI Grant Number JP23KJ1939.}



\begin{document}
	\begin{abstract}
		Recently, 
		\citeauthor{Iko2022} (\citeyear{Iko2022}) considered optimal constants and extremisers for the $2$-dimensional Dirac equation using the spherical harmonics decomposition. 
		Though its argument is valid in any dimensions $d \geq 2$, the case $d \geq 3$ remains open since it leads us to too complicated calculation: determining all eigenvalues and eigenvectors of infinite dimensional matrices.
		In this paper, we give  optimal constants and extremisers of smoothing estimates for the $3$-dimensional Dirac equation.
		In order to prove this, 
		we construct a certain orthonormal basis of spherical harmonics. 
		With respect to this basis, infinite dimensional matrices actually become block diagonal and so that eigenvalues and eigenvectors can be easily found. 
		As applications, we obtain the equivalence of the smoothing estimate for the Schr\"{o}dinger equation and the Dirac equation, 
		and improve a result by \citeauthor{BU2021} (\citeyear{BU2021}). 
	\end{abstract}
	\maketitle
	\section{Introduction}
	The Kato--Yajima smoothing estimates are one of the fundamental results in study of dispersive equations such as Schr\"odinger equations and Dirac equations, which were firstly observed by \citet{KY1989}, and have been studied by numerous researchers. 
	At first, we consider the following Schr\"odinger-type equation:
	\begin{equation} \label{eq:Schrodinger}
		\begin{cases}
			i \partial_t u(x, t) = \phi(\abs{D}) u(x, t) , & (x, t) \in \R^d \times \R , \\
			u(x, 0) = f(x) , & x \in \R^d , 
		\end{cases}
	\end{equation}
	where $\phi(\abs{D})$ denotes the Fourier multiplier operator whose symbol is $\phi(\abs{\variabledot})$, that is, 
	\begin{equation}
		\mathcal{F} \phi(\abs{D}) f = \phi(\abs{\xi}) \widehat{f}(\xi) .
	\end{equation}
	The function $\phi$ is called a dispersion relation.
	For example, \eqref{eq:Schrodinger} becomes the free Schr\"odinger equation if $\phi(r) = r^2$ and the relativistic Schr\"odinger equation if $\phi(r) = \sqrt{r^2 + m^2}$, respectively.
	The (global) smoothing estimate of the Schr\"odinger-type equation is expressed as
	\begin{equation}
		\int_{t \in \R}\int_{x \in \R^d}\abs{\psi(\abs{D}) e^{- it\phi(\abs{D})}f(x)}^2 w(\abs{x}) \, dx \, dt \leq C\norm{f}^2_{L^2(\R^d)} . \label{eq:smoothing Schrodinger}
	\end{equation}
	Here functions $w$ and $\psi$ are called spatial weight and smoothing function, respectively. 
	We write $C_{d}(w,\psi,\phi)$ for the optimal constant for the inequality \eqref{eq:smoothing Schrodinger}, in other words,
	\begin{equation}
		C_{d}(w,\psi,\phi) \coloneqq \sup_{ \substack{f\in L^2(\R^d) \\ 
				\norm{f}_2 = 1}}\int_{t \in \R}\int_{x \in \R^d}\abs{\psi(\abs{D})e^{- it\phi(\abs{D})}f(x)}^2 w(\abs{x}) \, dx \, dt. \label{eq:optimal Schrodinger}
	\end{equation}
	Since we are interested in explicit constants, here we clarify that the Fourier transform in this paper is defined by  
	\begin{equation}
		\hat{f}(\xi) \coloneqq \int_{x \in \R^d}f(x)e^{-ix\cdot\xi} \, dx .
	\end{equation}
	In this case, the Plancherel theorem states that 
	\begin{equation}
		\norm{\hat{f}}_{L^2(\R^d)}^2 = (2 \pi)^d \norm{f}_{L^2(\R^d)}^2 .
	\end{equation}
	
	\citet*{BSS2015} established the following abstract result for the Schr\"odinger-type equation:
	\begin{quotetheorem}[{\cite[Theorem 1.1]{BSS2015}}] \label{thm:Schrodinger}
		Let $d \geq 2$.
		Assume that $(w, \psi, \phi)$ is sufficiently nice (see Assumption \ref{assumption} below).
		Then we have
		\begin{equation}
			(2\pi)^{d-1} C_{d}(w,\psi,\phi) = \lambda^* \coloneqq \sup_{k\in\N}\sup_{r>0}\lambda_{k}(r), 
		\end{equation}
		where
		\begin{equation}
			\lambda_{k}(r) = \abs{\S^{d-2}}\frac{r^{d-1}\psi(r)^2}{\phi^{\prime}(r)}\int_{-1}^{1}F_{w}(r^2(1-t))p_{d,k}(t)(1-t^2)^{\frac{d-3}{2}} \, dt. \label{eq:lambda_k}
		\end{equation}
		Here $F_{w} \colon (0, \infty) \to \R$ is the function satisfying
		\begin{equation} \label{w_Fourier}
			F_w(\abs{\xi}^2 / 2) = \widehat{w(\abs{\variabledot})}(\xi) =
			\int_{x \in \R^d} w(\abs{x}) e^{- i x \cdot \xi} \, dx ,
		\end{equation}
		and $p_{d,k}$ is the Legendre polynomial of degree $k$ in $d$ dimensions, which may be defined in a number of ways, for example, via the Rodrigues formula,
		\begin{equation}
			(1-t^2)^{\frac{d-3}{2}}p_{d,k}(t) = (-1)^{k}\frac{\Gamma(\frac{d-1}{2})}{2^k\Gamma(k+\frac{d-1}{2})}
			\frac{d^k}{dt^k}(1-t^2)^{k+\frac{d-3}{2}}. \label{Legendre}
		\end{equation}
		Furthermore, extremisers exist if and only if there exists $k \in \N$ such that the Lebesgue measure of
		\begin{equation}
			\set{ r > 0 }{ \lambda_{k}(r) = \lambda^* }
		\end{equation}
		is non-zero.
	\end{quotetheorem}
	\begin{assumption} \label{assumption}
		Throughout the paper, we assume that $(w, \psi,\phi)$ satisfies the following conditions.
		\begin{itemize}
			\item
			The spatial weight $w \colon (0,\infty) \to [0, \infty)$ satisfies $w(\abs{\variabledot}) \in \mathcal{S}'(\R^d)$ in the sense that 
			\begin{equation}
				\mathcal{S}(\R^d) \ni \varphi \mapsto (w(\abs{\variabledot}), \varphi) \coloneqq \int_{x \in \R^d} w(\abs{x}) \varphi(x) \, dx
			\end{equation}
			is a tempered distribution. In addition, the Fourier transform of $w(\abs{\variabledot})$ is regular, which means that there exists $F_w \colon (0, \infty) \to \R$ satisfying
			\begin{equation}
				\int_{x \in \R^d} w(\abs{x}) \widehat{\varphi}(x) \, dx = 
				\int_{\xi \in \R^d} F_w(\abs{\xi}^2 /2) \varphi(\xi) \, d\xi
			\end{equation}
			holds for any $\varphi \in \mathcal{S}(\R^d)$.
			Furthermore, 
			\begin{equation}
				(0, \infty) \ni r \mapsto \int_{-1}^{1}F_{w}(r^2(1-t))p_{d,k}(t)(1-t^2)^{\frac{d-3}{2}} \, dt
			\end{equation}
			is continuous on $(0, \infty)$ for each $k \in \N$.
			\item 
			The smoothing function $\psi \colon (0, \infty) \to [0, \infty)$ is continuous.	
			\item The dispersion relation $\phi \colon (0, \infty) \to (0, \infty)$ is continuously differentiable and satisfies $\phi'(r) > 0$ for any $r > 0$.
		\end{itemize}
	\end{assumption}
	See \eqref{eq:type A}, \eqref{eq:type B}, \eqref{eq:type C} for typical examples for $(w, \psi, \phi)$.
	We note that $\lambda_{k} \colon (0, \infty) \to \R$ defined by \eqref{eq:lambda_k} is continuous under the assumption above.  
	Furthermore, it is known that $\lambda_k$ is actually non-negative (see \cite{BSS2015} for details).
	Now notice that if $(w,\psi,\phi)$ satisfies the assumption above, then $(w, \sqrt{r} \psi / \sqrt{\phi'}, r^2)$ also does.
	Hence, by the definition of $\{ \lambda_k \}_{k \in \N}$, Theorem \ref{thm:Schrodinger} immediately implies that 
	\begin{equation}
		C_{d}(w,\psi,\phi) = 2 C_{d}(w, \sqrt{r} \psi / \sqrt{\phi'}, r^2) 
	\end{equation}
	holds. 
	In particular, we have
	\begin{equation} 
		C_{d}(w,\psi,(r^2+m^2)^{1/2}) = 2 C_{d}(w, (r^2+m^2)^{1/4} \psi, r^2) ,
	\end{equation}
	this means that smoothing estimates and their optimal constants for relativistic Schr\"{o}dinger equations can be reduced to those for Schr\"{o}dinger equations.
	Hereinafter, we write $\phi_m(r) \coloneqq (r^2 + m^2)^{ 1/2 }$ for simplicity.
	
	Now we are going to discuss the free Dirac equation.	
	Let $d \geq 1$ and write $N \coloneqq 2^{\lfloor(d+1)/2\rfloor}$. 
	The $d$-dimensional free Dirac equation with mass $m \geq 0$ is given by 
	\begin{equation} \label{eq:Dirac}
		\begin{cases}
			i\partial_{t} u(x,t) = H_m u(x, t), & (x, t) \in \R^d \times \R , \\
			u(x,0) = f(x) , & x \in \R^d .
		\end{cases}
	\end{equation}
	Here $u$ and $f$ are $\C^N$-valued, and
	the Dirac operator $H_m$ is defined by
	\begin{equation}
		H_m \coloneqq \alpha\cdot D + m\beta = \sum_{j=1}^{d}\alpha_jD_j + m\beta ,
	\end{equation}
	where $\alpha_1, \alpha_2, \ldots, \alpha_d$, $\alpha_{d+1}=\beta$
	are $N\times N$ Hermitian matrices satisfying the anti-commutation relation $\alpha_j\alpha_k + \alpha_k\alpha_j = 2\delta_{jk}I_N$. 
	Note that we have $H_m^2 = ( - \Delta + m^2 ) \matrixI_N$.
	In this sense, the Dirac operator $H_m$ is similar to the relativistic Schr\"{o}dinger operator $(- \Delta + m^2)^{1/2}$.
	Hence, it is natural to conjecture that the following smoothing estimate for the Dirac equation,
	\begin{equation} \label{eq:smoothing Dirac}
		\int_{t \in \R}\int_{x \in \R^d}\abs{\psi(\abs{D})e^{-itH_m}f(x)}^2 w(\abs{x}) \, dx \, dt \leq C\norm{f}^2_{L^2(\R^d,\C^N)} ,
	\end{equation}
	is also similar to that for the relativistic Schr\"{o}dinger equation.
	Now let $\widetilde{C}_{d}(w,\psi,m)$ be the optimal constant of \eqref{eq:smoothing Dirac}, 
	and let $\widetilde{C}_{d, \rad}(w,\psi,m)$ be that with radial initial data, those are,
	\begin{gather} 
		\widetilde{C}_{d}(w,\psi,m) 
		\coloneqq \sup_{ \substack{ f\in L^2(\R^d,\C^N) \\ 
				\norm{f}_2 = 1}}\int_{t \in \R}\int_{x \in \R^d}\abs{\psi(\abs{D})e^{-itH_m}f(x)}^2 w(\abs{x}) \, dx \, dt , 
		\label{eq:optimal Dirac}\\
		\widetilde{C}_{d, \rad}(w,\psi,m) 
		\coloneqq \sup_{ \substack{f\in L^2(\R^d,\C^N) \\ 
				\norm{f}_2 = 1, f \text{:radial}} } \int_{t \in \R}\int_{x \in \R^d}\abs{\psi(\abs{D})e^{-itH_m}f(x)}^2 w(\abs{x}) \, dx \, dt .
		\label{eq:optimal Dirac radial}
	\end{gather}
	Recently, \citet{Iko2022} and \citet{IkS2023} studied \eqref{eq:optimal Dirac} for $d = 2$ and \eqref{eq:optimal Dirac radial} for arbitrary $d \geq 2$. They obtained the following results, which are analogous to Theorem \ref{thm:Schrodinger}.
	\begin{quotetheorem}[{\cite[Theorems 2.1, 2.2]{Iko2022}}]\label{thm:2D Dirac intro}
		Let $d \geq 2$ and write
		\begin{equation}
			\Diraclambda_{k}(r) 
			\coloneqq \frac{1}{2} \mleft( \lambda_{k}(r) + \lambda_{k+1}(r) + \frac{m}{\sqrt{r^2+m^2}} \abs{ \lambda_{k}(r) - \lambda_{k+1}(r) } \mright) ,
		\end{equation}
		where $\lambda_{k}$ is that given by \eqref{eq:lambda_k} associated with $(w, \psi, \phi_m)$.
		Then we have
		\begin{equation} \label{eq:Dirac 2D} 
			2 \pi \widetilde{C}_{2}(w,\psi,m) = \Diraclambda^* \coloneqq \sup_{ k \in \N }\sup_{r>0}\Diraclambda_{k}(r) 
		\end{equation}
		if $d = 2$, and
		\begin{equation} 
			\widetilde{C}_{d}(w,\psi,m) \leq C_{d}(w,\psi,\phi_m) = 2 C_{d}(w, \phi_m^{1/2} \psi, r^2) 
		\end{equation}	
		whenever $d \geq 2$.
		Furthermore, in the case $d=2$, extremisers exist if and only if there exists $k \in \N$ such that the Lebesgue measure of
		\begin{equation}
			\set{ r > 0 }{ \Diraclambda_{k}(r) = \Diraclambda^* }
		\end{equation}
		is non-zero.
	\end{quotetheorem}
	\begin{quotetheorem}[{\cite[Theorem 2.5]{IkS2023}}]\label{thm:radial Dirac}
		Let $d \geq 2$ and write
		\begin{equation}
			\Diraclambda_\rad(r) \coloneqq \frac{1}{2}\mleft( \lambda_0(r) + \lambda_1(r) + \frac{m^2}{r^2 + m^2}( \lambda_0(r) - \lambda_1(r) ) \mright), \label{eq:lambda radial Dirac} 
		\end{equation}
		where $\lambda_{k}$ is that given by \eqref{eq:lambda_k} associated with $(w, \psi, \phi_m)$.
		Then we have
		\begin{equation} 
			(2\pi)^{d-1} \widetilde{C}_{d, \rad}(w,\psi,m) = \Diraclambda_\rad^* \coloneqq \sup_{r>0} \Diraclambda_\rad(r) .
		\end{equation}
		Furthermore, extremisers exist if and only if the Lebesgue measure of
		\begin{equation}
			\set{ r > 0 }{ \Diraclambda_\rad(r) = \Diraclambda_\rad^* }
		\end{equation}
		is non-zero.
	\end{quotetheorem}
	In this paper, we show that the identity \eqref{eq:Dirac 2D} also holds in the physically most important case $d=3$.
	\begin{theorem}\label{thm:3D Dirac intro}
		Let $d = 3$. 
		Then we have
		\begin{equation} 
			(2\pi)^2 \widetilde{C}_{3}(w,\psi,m) = \Diraclambda^* = \sup_{ k \in \N }\sup_{r>0}\Diraclambda_{k}(r) .
		\end{equation}
		Furthermore, extremisers exist if and only if there exists $k \in \N$ such that the Lebesgue measure of
		\begin{equation}
			\set{ r > 0 }{ \Diraclambda_{k}(r) = \Diraclambda^* }
		\end{equation}
		is non-zero.
	\end{theorem}
	\begin{remark}
		\citet*{BRS2020} established smoothing estimates and obtained optimal constants in more abstract setting.
		Let $\Hilbert$ be a Hilbert space, $\Hermiteop$ be a self-adjoint operator on $\Hilbert$ whose spectrum $\sigma(\Hermiteop)$ is purely absolutely continuous, and $\{ \spectral{H}(r) \}_{r \in \R}$ denotes the spectral family of $\Hermiteop$. 
		Consider the following abstract Schr\"{o}dinger equation:
		\begin{equation}
			\begin{cases}
				i \partial_t u(t) = \phi(\Hermiteop) u(t) , & t \in \R , \\
				u(0) = f , & f \in \Hilbert .
			\end{cases}
		\end{equation}
		\citet[Theorem 3.8]{BRS2020} states that the optimal constant of the smoothing estimate
		\begin{equation}
			\int_{t \in \R} \norm{\psi(\Hermiteop) e^{ - i t \phi(\Hermiteop) } f}_{\adjoint{\mathcal{X}}}^2 \, dt \leq C \norm{f}_{\Hilbert}^2
		\end{equation}
		is given by
		\begin{equation}
			C = 2 \pi \sup_{r \in \sigma(\Hermiteop)} \frac{ ( \psi(r) )^2 }{ \abs{ \phi'(r) } } \norm{ A_{\Hermiteop}(r) }_{ \mathcal{X} \to \adjoint{\mathcal{X}} }^2 ,
		\end{equation}
		where:
		\begin{itemize}
			\item $\psi$ and $\phi$ are sufficiently nice real-valued functions (see \citet[Assumption 3.2]{BRS2020} for details).
			\item $\mathcal{X} \subset \Hilbert$ is a dense subspace with a stronger norm $\norm{\variabledot}_{\mathcal{X}}$ and so that $\Hilbert \subset \adjoint{\mathcal{X}}$ via the canonical embedding
			\begin{equation}
				\Hilbert \ni f \mapsto ( \varphi \in \mathcal{X} \mapsto \innerproduct{ f }{ \varphi }_{\Hilbert} ) \in \adjoint{\mathcal{X}} .
			\end{equation}
			\item $A_{\Hermiteop}(r) \colon \mathcal{X} \to \adjoint{\mathcal{X}}$ is the spectral derivative
			\begin{equation}
				\innerproduct{ A_{\Hermiteop}(r) f }{ g }_{\adjoint{\mathcal{X}}, \mathcal{X}} = \frac{d}{dr} \innerproduct{ \spectral{H}(r) f }{ g }_{\Hilbert} .
			\end{equation}
		\end{itemize}
		In particular, for a spatial weight satisfying $0 < w \in L^\infty(\R^d)$, letting
		\begin{equation}
			\Hilbert = L^2(\R^d) , 
			\quad \mathcal{X} = L^2_{1/w}(\R^d) , \quad
			\Hermiteop = \abs{ \nabla }
		\end{equation}
		and using $\adjoint{\mathcal{X}} \cong L^2_w(\R^d)$, we obtain 
		\begin{equation}
			C_{d}(w,\psi,\phi) = 
			2 \pi \sup_{r \geq 0} \frac{ ( \psi(r) )^2 }{ \abs{ \phi'(r) } } \norm{ A_{\abs{ \nabla }}(r) }_{ L^2_{1/w}(\R^d) \to L^2_w(\R^d) }^2 .
		\end{equation}
		Similarly, 
		letting
		\begin{equation}
			\Hilbert = L^2(\R^d, \C^N) , 
			\quad \mathcal{X} = L^2_{1/w}(\R^d, \C^N) , \quad
			\Hermiteop = H_m
		\end{equation}
		gives us
		\begin{equation}
			\widetilde{C}_{d}(w,\psi,m) = 
			2 \pi \sup_{\abs{r} \geq m} ( \psi( \sqrt{ r^2 - m^2 } ) )^2  \norm{ A_{H_m}(r) }_{ L^2_{1/w}(\R^d, \C^N) \to L^2_w(\R^d, \C^N) }^2 .
		\end{equation}
		However, this method is inapplicable for unbounded weights such as $w(x) = \abs{x}^{-s}$ unlike Theorems \ref{thm:Schrodinger} and \ref{thm:3D Dirac intro}, since $L^2_{1/w} \not \subset L^2$. 
	\end{remark}
	Combining Theorems \ref{thm:Schrodinger}, \ref{thm:2D Dirac intro}, \ref{thm:3D Dirac intro}, and the trivial inequality
	\begin{equation}
		\frac{1}{2} \max\{ \lambda_k(r), \lambda_{k+1}(r) \} \leq \Diraclambdak{k}(r) \leq \max\{ \lambda_k(r), \lambda_{k+1}(r) \} ,
	\end{equation}
	we obtain the equivalence of smoothing estimates for the Schr\"{o}dinger, relativistic Schr\"{o}dinger, and Dirac equations when $d = 2, 3$.
	\begin{corollary} \label{cor:Dirac Schrodinger equivalence}
		Let $d \geq 2$.
		Then we have
		\begin{equation} 
			\widetilde{C}_{d}(w,\psi,m) \leq C_{d}(w,\psi,\phi_m) = 2 C_{d}(w, \phi_m^{1/2} \psi, r^2).
		\end{equation}
		Furthermore, when $d = 2, 3$, we also have
		\begin{equation} 
			C_{d}(w, \phi_m^{1/2} \psi, r^2) = \frac{1}{2} C_{d}(w,\psi,\phi_m) \leq \widetilde{C}_{d}(w,\psi,m).
		\end{equation}
	\end{corollary}
	For the Schr\"{o}dinger equation, it is classically known that the smoothing estimate holds in the following cases:
	\begin{alignat}{8}
		&d \geq 3,  \quad && s \geq 2 , \quad && ( w(r), \psi(r) ) = ( &&(1+r^2)^{-s/2}, \quad&&(1+r^2)^{1/4} &&) ,
		\label{eq:type A} \tag{A}\\
		&d \geq 2,  \quad && 1 < s < d, \quad&& ( w(r), \psi(r) ) = ( &&r^{-s}, \quad&&r^{(2-s)/2} &&) , 
		\label{eq:type B} \tag{B} \\
		&d \geq 2,  \quad&& s > 1 , \quad&& ( w(r), \psi(r) ) = ( &&(1+r^2)^{-s/2}, \quad&&r^{1/2} &&) .
		\label{eq:type C} \tag{C}
	\end{alignat}
	The case \eqref{eq:type A} is given by \citet[Theorem 2]{KY1989}.
	\citet[Theorem 1, Remarks (a)]{KY1989} also proved the case \eqref{eq:type B} with $d=2$, $1 < s < 2$ and $d \geq 3$, $1 < s \leq 2$.
	See \citet[Theorem 1.1]{Sug1998} for the case \eqref{eq:type B} with the full range $1 < s < d$.
	The case \eqref{eq:type C} is by \citet[Theorem 1.(b)]{BK1992} ($d \geq 3$) and \citet[Theorem 1.1]{Chi2002} ($d \geq 2$).
	Furthermore, the ranges $s \geq 2$ in \eqref{eq:type A}, $1 < s < d$ in \eqref{eq:type B}, and $s > 1$ in \eqref{eq:type C} are sharp; 
	see \citet[Theorem 2.1.(b), Theorem 2.2.(b)]{Wal1999} for \eqref{eq:type A}, 
	\citet[Theorem 2]{Vil2001} for \eqref{eq:type B}, and \citet[Theorem 2.14.(b)]{Wal2000} for \eqref{eq:type C}.  
	
	\citet*{BS2017,BSS2015} determined the explicit values of the optimal constant and the existence of extremisers in these cases as follows:
	\begin{quotetheorem}[{\cite[Theorems 1.6, 1.7]{BS2017}, \cite[Theorem 1.4]{BSS2015}}] \label{thm:Schrodinger explicit value}\phantom{a}
		\begin{description}
			\item[\rm{\cite[Theorem 1.7]{BS2017}}]
			In the case \eqref{eq:type A} with $s = 2$, we have
			\begin{equation}
				C_{d}((1+r^2)^{-1} , (1+r^2)^{1/4} , r^2) = \begin{cases}
					\pi , & d = 3 , \\
					\pi / 2 , & d \geq 5 ,
				\end{cases} 
			\end{equation}
			and extremisers do not exist.
			\item[\rm{\cite[Theorem 1.6]{BS2017}}]
			In the case \eqref{eq:type B}, we have
			\begin{equation}
				C_{d}( r^{-s} , r^{(2-s)/2} , r^2) = 2^{1-s} \pi \frac{ \Gamma(s-1) \Gamma((d-s)/2) }{ ( \Gamma(s/2) )^2 \Gamma( (d+s)/2 - 1 ) } , 
			\end{equation}
			and $f \in L^2(\R^d) \setminus \{0\}$ is an extremiser if and only if $f$ is radial.
			In particular, we have
			\begin{equation}
				C_{d}( r^{-2} , 1 , r^2) = \frac{\pi}{d-2} 
			\end{equation}
			when $d \geq 3$ and $s = 2$, 
			which recovers \citeauthor{Sim1992}'s result \cite[(3)]{Sim1992}.
			\item[\rm{\cite[Corollary 1.5]{BSS2015}}]
			In the case \eqref{eq:type C}, we have
			\begin{equation}
				C_{d}( (1+r^2)^{-s/2} , r^{1/2} , r^2) = \pi^{1/2} \frac{ \Gamma( (s-1)/2 ) }{ 2 \Gamma(s) }
			\end{equation}
			whenever $d \geq 3$, and extremisers do not exist.
			In particular, we have
			\begin{equation}
				C_{d}( (1+r^2)^{-1} , r^{1/2} , r^2) = \frac{\pi}{2}
			\end{equation}
			when $d \geq 3$ and $s = 2$, which recovers \citeauthor{Sim1992}'s another result \cite[(2)]{Sim1992}.
		\end{description}
	\end{quotetheorem}
	Note that one can interpret these results for Schr\"{o}dinger equations as for relativistic Schr\"{o}dinger and Dirac equations via Corollary \ref{cor:Dirac Schrodinger equivalence}. 
	For example, the result for the case \eqref{eq:type A} is equivalent to
	\begin{equation}
		C_{d}((1+r^2)^{-1},1,(1+r^2)^{1/2}) = \begin{cases}
			2 \pi , & d = 3 , \\
			\pi , & d \geq 5 ,
		\end{cases}  
	\end{equation}
	this gives the optimal constant of the inequality for the relativistic Schr\"{o}dinger equation established by \citet[Theorem 3A, Corollary 3.1]{BN1997}.
	We can also improve a recent result for the Dirac equation established \citet[Theorem 7.1]{BU2021}, which states that the smoothing estimate for the Dirac equation holds if 
	\begin{alignat}{8}
		&d = 3,   \quad && s > 2 , \quad && m > 0 , \quad &&  ( w(r), \psi(r) ) = ( &&(1+r^2)^{-s/2}, \quad&& 1 &&) .
	\end{alignat}
	\begin{corollary} \label{cor:type A Dirac}
		The smoothing estimate for the Dirac equation holds in the following case:
		\begin{alignat}{8}
			&d \geq 3,   \quad && s \geq 2 , \quad && m \geq 0 , \quad &&  ( w(r), \psi(r) ) = ( &&(1+r^2)^{-s/2}, \quad&& \phi_m(r)^{-1/2} (1 + r^2)^{1/4} &&) ,
			\label{eq:type A Dirac} \tag{$\widetilde{\textrm{A}}$}
		\end{alignat}
		Here, the range $s \geq 2$ in \eqref{eq:type A Dirac} is sharp.
		Furthermore, we have
		\begin{equation}
			\pi \leq \widetilde{C}_{3}((1+r^2)^{-1} , \phi_m(r)^{-1/2} (1 + r^2)^{1/4} , m) \leq 2 \pi 
		\end{equation}
		when $d = 3$, and 
		\begin{equation}
			\widetilde{C}_{d}((1+r^2)^{-1} , \phi_m(r)^{-1/2} (1 + r^2)^{1/4} , m) \leq \pi
		\end{equation}
		when $d \geq 5$.
	\end{corollary}
	Moreover, in the cases \eqref{eq:type B} and \eqref{eq:type C}, we can determine the explicit value of $\widetilde{C}_{d}(w,\phi_{m}^{-1/2} \psi,m) $ using Theorems \ref{thm:radial Dirac} and \ref{thm:3D Dirac intro}.
	\begin{theorem} \label{thm:type B Dirac}
		The smoothing estimate for the Dirac equation holds in the following case:
		\begin{alignat}{8}
			&d \geq 2, \quad && 1 < s < d, \quad&& m \geq 0 , \quad && ( w(r), \psi(r) ) = ( &&r^{-s}, \quad&&\phi_{m}(r)^{-1/2} r^{(2-s)/2} &&) .
			\label{eq:type B Dirac} \tag{$\widetilde{\textrm{B}}$}
		\end{alignat}			
		Here, the range $1 < s < d$ in \eqref{eq:type B Dirac} is sharp.
		Furthermore, we have
		\begin{align}
			\widetilde{C}_{d}(r^{-s},\phi_{m}(r)^{-1/2} r^{(2-s)/2},m) 
			&= \begin{cases}
				2 C_{d}( r^{-s} , r^{(2-s)/2} , r^2) , & d \geq 2, \, m > 0 , \\
				2 \mleft( 1 - \dfrac{s - 1}{ d + s - 2 }  \mright) C_{d}( r^{-s} , r^{(2-s)/2} , r^2) , & d = 2, 3, \, m = 0 
			\end{cases} \\
			&= \begin{cases}
				2^{2-s} \pi \dfrac{ \Gamma(s-1) \Gamma((d-s)/2) }{ ( \Gamma(s/2) )^2 \Gamma( (d+s)/2 - 1 ) }, &  d \geq 2, \, m > 0 , \\
				\mleft( 1 - \dfrac{s - 1}{ d + s - 2 }  \mright) 2^{2-s} \pi \dfrac{ \Gamma(s-1) \Gamma((d-s)/2) }{ ( \Gamma(s/2) )^2 \Gamma( (d+s)/2 - 1 ) },  & d = 2, 3, \, m = 0 ,
			\end{cases}
		\end{align}
		and extremisers exist if $d=2, 3$ and $m = 0$, and do not if $d=2, 3$ and $m > 0$.
	\end{theorem}
	\begin{theorem} \label{thm:type C Dirac}
		The smoothing estimate for the Dirac equation holds in the following case:
		\begin{alignat}{8}
			&d \geq 2, \quad && s>1, \quad&& m \geq 0 , \quad && ( w(r), \psi(r) ) = ( &&(1+r^2)^{-s/2}, \quad&&\phi_{m}(r)^{-1/2} r^{1/2} &&) .
			\label{eq:type C Dirac} \tag{$\widetilde{\textrm{C}}$}
		\end{alignat}			
		Here, the range $s > 1$ in \eqref{eq:type C Dirac} is sharp.
		Furthermore, we have
		\begin{align}
			\widetilde{C}_{d}((1+r^2)^{-s/2},\phi_{m}(r)^{-1/2} r^{1/2},m) 
			&= 
			2 C_{d}((1+r^2)^{-s/2},r^{1/2},r^2)  \\
			&= \pi^{1/2} \frac{ \Gamma( (s-1)/2 ) }{ \Gamma(s) } 
		\end{align}
		whenever $d \geq 3$.
	\end{theorem}
	Recall that in the case \eqref{eq:type B}, the smoothing estimate for the Schr\"{o}dinger equation has extremisers.
	Nevertheless, in the case \eqref{eq:type B Dirac} with $d=2,3$ and $m > 0$, the smoothing estimate for the Dirac equation does not have extremisers.
	Which means that the existence of extremisers for the smoothing estimate for the Dirac equation with $(w, \phi_{m}^{-1/2} \psi)$ is 
	not equivalent to that for the Schr\"{o}dinger equation with $(w, \psi)$, even though the smoothing estimates themselves are equivalent.
	
	In addition, the case \eqref{eq:type B Dirac} with $m > 0$ shows that the first inequality in Corollary \ref{cor:Dirac Schrodinger equivalence} is sharp. 
	Furthermore, in the case \eqref{eq:type B Dirac} with $d=2, 3$ and $m = 0$, we have
	\begin{equation}
		\lim_{s \uparrow d} \frac{ \widetilde{C}_{d}(r^{-s},\phi_{m}(r)^{-1/2} r^{(2-s)/2},m)  }{ C_{d}( r^{-s} , r^{(2-s)/2} , r^2) }
		= \lim_{s \uparrow d} 2 \mleft( 1 - \frac{s - 1}{ d + s - 2 }  \mright)  
		= 1 ,
	\end{equation}
	which shows that the second inequality in Corollary \ref{cor:Dirac Schrodinger equivalence} is also sharp. 
	\subsection*{Organization of the paper}
	In Section \ref{section:preliminaries}, we introduce some notation and basic facts as preliminaries.
	In particular, Lemma \ref{lem:Sf norm decomposition Dirac}, which follows from the spherical harmonics decomposition (Theorem \ref{thm:spherical harmonics decomposition}) and the Funk--Hecke theorem (Theorem \ref{thm:Funk-Hecke}), plays an important role.
	Though Lemma \ref{lem:Sf norm decomposition Dirac} holds for any orthonormal bases of spherical harmonics, we will choose a certain basis to avoid a tedious calculation.
	To illustrate our idea, we give a simplified proof of Theorem \ref{thm:2D Dirac intro} (the optimal constant in the case $d = 2$) in Section \ref{section:proof for 2D}.
	We will see that expressing the spherical harmonics decomposition of $f \in L^2(\R^2, \C^2)$ as
	\begin{equation}
		f(\xi) = r^{-1/2} \sum_{k=-\infty}^{\infty}
		\frac{1}{\sqrt{2\pi}} \begin{pmatrix}
			e^{i k \theta} & 0 \\ 
			0 & e^{i (k+1) \theta}
		\end{pmatrix} \fk{k}(r) , \quad \xi = (r \cos \theta, r \sin \theta) 
	\end{equation}
	instead of as the usual form (which is used by \citet{Iko2022}), 
	\begin{equation}
		f(\xi) = r^{-1/2} \sum_{k=-\infty}^{\infty}
		\frac{1}{\sqrt{2\pi}} e^{i k \theta} \fk{k}(r) , \quad \xi = (r \cos \theta, r \sin \theta)  ,
	\end{equation}
	simplifies the proof significantly.
	The main advantage of using
	\begin{equation}
		\matrixEk{k}(\theta)
		\coloneqq 
		\frac{1}{\sqrt{2\pi}} \begin{pmatrix}
			e^{i k \theta} & 0 \\ 
			0 & e^{i (k+1) \theta}
		\end{pmatrix} 
	\end{equation}
	is the following identity:
	\begin{align}
		\mleft( \sigma_1 \cos{\theta}
		+ 
		\sigma_2 \sin{\theta} 
		\mright)
		\matrixEk{k}(\theta)
		&= 
		\matrixEk{k}(\theta) \sigma_1 ,\\
		\sigma_3 \matrixEk{k}(\theta) &= \matrixEk{k}(\theta) \sigma_3 ,
	\end{align} 
	where $\sigma_1, \sigma_2, \sigma_3$ are the Pauli matrices.
	In Section \ref{section:proof for 3D}, we first 
	give a certain expression of the spherical harmonic 
	decomposition of $f \in L^2(\R^3, \C^2)$, 
	\begin{equation}
		f(\xi) =  r^{-1} \sum_{k = 0}^{\infty} \sumnk 
		\matrixEkn{k}{n}(\theta, \varphi)
		\fkn{k}{n}(r) ,
		\quad \xi = (r \sin \theta \cos \varphi , r \sin \theta \sin \varphi, \cos \theta) , 
	\end{equation}
	where $\matrixEkn{k}{n}(\theta, \varphi)$ satisfies
	\begin{align}
		\sigma_1 \otimes \mleft( \sigma_1 \sin{\theta} \cos{\varphi} 
		+ 
		\sigma_2 \sin{\theta} \sin{\varphi} 
		+ \sigma_3 \cos{\theta} \mright) 
		\matrixEkn{k}{n}(\theta, \varphi)
		&= \matrixEkn{k}{n}(\theta, \varphi) 
		(\sigma_1 \otimes \matrixI) , \\
		(\sigma_3 \otimes \matrixI) \matrixEkn{k}{n}(\theta, \varphi) &= \matrixEkn{k}{n}(\theta, \varphi) (\sigma_3 \otimes \sigma_3) .	 
	\end{align}
	Once it is established, the proof of our main result Theorem \ref{thm:3D Dirac intro} is similar to that of Theorem \ref{thm:2D Dirac intro} given in Section \ref{section:proof for 2D}.
	Finally, we prove Theorems \ref{thm:type B Dirac} and \ref{thm:type C Dirac} in Section \ref{section:explicit value}.
	\section{Preliminaries} \label{section:preliminaries}
	We define the linear operator $\DiracS \colon L^2(\R^d, \C^N) \to L^2(\R^{d+1}, \C^N)$ by 
	\begin{equation}
		(\DiracS f)(x,t) \coloneqq w(\abs{x})^{1/2} \int_{\xi \in \R^d} e^{ix \cdot \xi}\psi(\abs{\xi})e^{-itA_{\xi}}f(\xi)\, d\xi ,
	\end{equation}
	where
	\begin{equation}
		A_\xi \coloneqq \alpha\cdot \xi + m\beta = \sum_{j=1}^{d}\alpha_j \xi_j + m\beta .
	\end{equation}
	Notice that the smoothing estimate \eqref{eq:smoothing Dirac} is equivalent to 
	\begin{equation}
		\norm{\DiracS \hat{f}}_{L^2(\R^{d+1}, \C^N)}^2 \leq C \norm{f}_{L^2(\R^d)}^2 = (2 \pi)^d C \norm{\hat{f}}_{L^2(\R^d, \C^N)}^2 ,
	\end{equation}
	and so that $\norm{\DiracS}_{L^2(\R^d, \C^N) \to L^2(\R^{d+1}, \C^N)}^2 = (2 \pi)^d C_{d}(w,\psi,m)$.
	For simplicity, hereinafter we write $\norm{\DiracS} \coloneqq \norm{\DiracS}_{L^2(\R^d, \C^N) \to L^2(\R^{d+1}, \C^N)}$.
	
	We note 
	that $\norm{\DiracS}$ is independent of the choice of $\alpha_1, \alpha_2, \ldots, \alpha_d$, $\alpha_{d+1} = \beta$.
	To see this, let $\{ \alpha_j^{(1)} \}_{j=1}^{d+1}$ and $\{ \alpha_j^{(2)} \}_{j=1}^{d+1}$ be $N\times N$ Hermitian matrices satisfying the anti-commutation relation.
	By the so-called fundamental theorem on Dirac gamma matrices, there exists a unitary matrix $U$ satisfying
	$\alpha_j^{(2)} = U^{-1} \alpha_j^{(1)} U$. 
	Therefore, we have
	\begin{align}
		\DiracS^{(2)} f(x,t) 
		&= w(\abs{x})^{1/2} \int_{\R^d} e^{ix \cdot \xi}\psi(\abs{\xi})e^{-itA_{\xi}^{(2)}}f(\xi)\, d\xi \\
		&= U^{-1} w(\abs{x})^{1/2} \int_{\R^d} e^{ix \cdot \xi}\psi(\abs{\xi})e^{-itA_{\xi}^{(1)}} Uf(\xi)\, d\xi \\
		&= U^{-1} \DiracS^{(1)} Uf(x,t) 
	\end{align}
	and so that
	\begin{equation}
		\norm{\DiracS^{(2)} f}_{L^2(\R^{d+1}, \C^N)} = \norm{\DiracS^{(1)} Uf}_{L^2(\R^{d+1}, \C^N)},
	\end{equation}
	which shows the independence.
	In this paper, we will use
	\begin{equation}
		\alpha_1 = \sigma_1 =
		\begin{pmatrix}
			0 & 1 \\
			1 & 0
		\end{pmatrix}, \quad
		\alpha_2 = \sigma_2 =
		\begin{pmatrix}
			0 & -i \\
			i & 0
		\end{pmatrix}, \quad
		\beta = \sigma_3 =
		\begin{pmatrix}
			1 & 0 \\
			0 & -1
		\end{pmatrix}
	\end{equation}
	in the case $d=2$, and 
	\begin{equation}
		\alpha_j = \sigma_1 \otimes \sigma_j =
		\begin{pmatrix}
			\matrixO & \sigma_j\\
			\sigma_j & \matrixO
		\end{pmatrix}, \quad 
		\beta = \sigma_3 \otimes \matrixI =
		\begin{pmatrix}
			\matrixI & \matrixO \\
			\matrixO & - \matrixI
		\end{pmatrix} 
	\end{equation}
	in the case $d=3$, where
	\begin{equation}
		\matrixO = \begin{pmatrix}
			0 & 0 \\
			0 & 0
		\end{pmatrix} , \quad 
		\matrixI = 
		\begin{pmatrix}
			1 & 0 \\
			0 & 1
		\end{pmatrix} .
	\end{equation}
	
	In order to compute $\norm{\DiracS f}_{L^2(\R^{d+1}, \C^N)}$, we will use the spherical harmonics decomposition (Theorem \ref{thm:spherical harmonics decomposition}) and the Funk--Hecke theorem (Theorem \ref{thm:Funk-Hecke}). 
	For each $k \in \N$, let $\HP_{k}(\R^d)$ and $\{ \Pkn{k}{n} \}_{1 \leq n \leq d_k}$ be the space of homogeneous harmonic polynomials of degree $k$ on $\R^d$ and its orthonormal basis, respectively.
	Here, 
	the inner product of $P, Q \in \HP_{k}(\R^d)$ is defined by
	\begin{equation}
		\innerproduct{P}{Q}_{\HP_{k}(\R^d)}
		\coloneqq \innerproduct{ \restr{P}{\S^{d-1}} }{ \restr{Q}{\S^{d-1}} }_{L^2(\S^{d-1})}
		= \int_{ \theta \in \S^{d-1} } P(\theta) \conjugate{Q(\theta)} \, d\sigma(\theta) ,
	\end{equation}
	as usual.
	The spherical harmonics decomposition and the Funk--Hecke theorem are as follows:
	\begin{quotetheorem} \label{thm:spherical harmonics decomposition}
		For any $f \in L^2(\R^d)$, there uniquely exists $\{ \fkn{k}{n} \}_{ k \in \N, 1 \leq n \leq d_k } \subset L^2(\positiveR)$ satisfying
		\begin{gather}
			f(\xi) = \abs{\xi}^{-(d-1)/2} \sum_{k=0}^{\infty} \sum_{n=1}^{d_k} \Pkn{k}{n}(\xi / \abs{\xi}) \fkn{k}{n}(\abs{\xi}) , \label{eq:spherical harmonics decomposition} \\
			\norm{f}_{L^2(\R^d)}^2 = \sum_{k=0}^{\infty} \sum_{n=1}^{d_k}  \norm{\fkn{k}{n}}_{L^2(\positiveR)}^2 . \label{eq:spherical harmonics decomposition norm}
		\end{gather}
		Conversely, for any $\{ \fkn{k}{n} \}_{ k \in \N, 1 \leq n \leq d_k } \subset L^2(\positiveR)$ satisfying
		\begin{equation}
			\sum_{k=0}^{\infty} \sum_{n=1}^{d_k}  \norm{\fkn{k}{n}}_{L^2(\positiveR)}^2 < \infty ,
		\end{equation}
		the function $f$ given by \eqref{eq:spherical harmonics decomposition} is in $L^2(\R^d)$ and \eqref{eq:spherical harmonics decomposition norm} holds.
	\end{quotetheorem}
	\begin{quotetheorem} \label{thm:Funk-Hecke}
		Let $d \geq 2$, $k \in \N$, $F \in L^1 ( [-1, 1], (1-t^2)^{(d-3)/2} \, dt )$ and write
		\begin{equation} \label{eq:mu in Funk-Hecke}
			\mu_{k}[F] \coloneqq \measure{\S^{d-2}} \int_{-1}^1 F(t) p_{d, k}(t) 
			(1-t^2)^{\frac{d-3}{2}} \, dt.
		\end{equation} 
		Here recall that $p_{d, k}$ denotes the Legendre polynomial of degree $k$ in $d$ dimensions; 
		see \eqref{Legendre}.
		Then, for any $P \in \HP_k(\R^d)$ and $\omega \in \S^{d-1}$, we have
		\begin{equation} \label{eq:Funk-Hecke}
			\int_{\theta \in \S^{d-1}} F(\theta \cdot \omega) P(\theta) \, d\sigma(\theta) 
			= \mu_{k}[F] P(\omega) . 
		\end{equation}
	\end{quotetheorem}
	Note that the function $\lambda_k$ defined in \eqref{eq:lambda_k} satisfies
	\begin{equation}
		\lambda_{k}(r) 
		= \frac{r^{d-1}\psi(r)^2}{\abs{\phi^{\prime}(r)}} \mu_{k}[F_w( r^2 (1 - \variabledot) ) ] , 
	\end{equation}
	in other words, it satisfies 
	\begin{equation}
		\frac{r^{d-1}\psi(r)^2}{\abs{\phi^{\prime}(r)}} \int_{\theta \in \S^{d-1}} F_w( r^2(1 - \theta \cdot \omega) ) P(\theta) \, d\sigma(\theta) 
		= \lambda_k(r) P(\omega) 
	\end{equation}
	for each $k \in \N$, $P \in \HP_k(\R^d)$ and $\omega \in \S^{d-1}$.
	
	Using these facts, 
	we can obtain the following important lemma.
	\begin{quotelemma} \label{lem:Sf norm decomposition Dirac}
		Let $f \in L^2(\R^d, \C^N)$ and define $\fpm \in L^2(\R^d, \C^N)$ by
		\begin{equation}
			\fpm(\xi) \coloneqq \frac{1}{2} \mleft( f(\xi) \pm \frac{1}{\phi_{m}(\abs{\xi})} \bigg( m \beta f(\xi) + \sum_{j=1}^{d} \alpha_j \xi_j f(\xi) \bigg) \mright) .
		\end{equation}
		If
		\begin{equation}
			\fpm(\xi) = \abs{\xi}^{-(d-1)/2} \sum_{k=0}^{\infty} \sum_{n=1}^{d_k} \Pkn{k}{n}(\xi / \abs{\xi}) \fknpm{k}{n}(\abs{\xi}) ,
		\end{equation}
		then we have
		\begin{equation}
			\norm{\DiracS f}_{L^2(\R^{d+1}, \C^N)}^2 = 2 \pi \sum_{k=0}^{\infty} \sum_{n=1}^{d_k} \int_{0}^{\infty} \lambda_k (r) ( \abs{\fknp{k}{n}(r)}^2 + \abs{\fknm{k}{n}(r)}^2 )  \, dr ,
		\end{equation}
		where $\lambda_k$ is that given by \eqref{eq:lambda_k} associated with $(w, \psi, \phi_m)$.
	\end{quotelemma}	
	We omit the proof of Lemma \ref{lem:Sf norm decomposition Dirac}. See \cite[Proof of Theorem 2.1]{Iko2022}.
	
	\section{In the case \texorpdfstring{$d=2$}{d=2}} \label{section:proof for 2D}
	In this section, we use the usual Pauli matrices for $\{ \alpha_j \}_{j=1}^{3}$:
	\begin{equation}
		\alpha_1 = \sigma_1 =
		\begin{pmatrix}
			0 & 1 \\
			1 & 0
		\end{pmatrix}, \quad
		\alpha_2 = \sigma_2 =
		\begin{pmatrix}
			0 & -i \\
			i & 0
		\end{pmatrix}, \quad
		\beta = \sigma_3 =
		\begin{pmatrix}
			1 & 0 \\
			0 & -1
		\end{pmatrix}.
	\end{equation}
	At first we prove the following lemma:	
	\begin{lemma} \label{lem:lemma 2D}
		Let
		\begin{equation}
			\matrixEk{k}(\theta)
			\coloneqq 
			\frac{1}{\sqrt{2\pi}} \begin{pmatrix}
				e^{i k \theta} & 0 \\ 
				0 & e^{i (k+1) \theta}
			\end{pmatrix} .
		\end{equation}
		Then the following hold:
		\begin{eqenumerate}
			\eqitem \label{item:spherical harmonics decomposition 2D}
			For any $f \in L^2(\R^2, \C^2)$, there uniquely exists $\{ \fk{k} \} \subset L^2(\R_{>0}, \C^2)$ satisfying
			\begin{gather}
				f(\xi) = r^{-1/2} \sum_{k=-\infty}^{\infty} \matrixEk{k}(\theta) \fk{k}(r) , \quad \xi = (r \cos \theta, r \sin \theta) , 
				\label{eq:spherical harmonics decomposition 2D} \tag{\ref{item:spherical harmonics decomposition 2D}.i} \\
				\norm{f}_{L^2(\R^2, \C^2)}^2 = \sum_{k=-\infty}^{\infty} \norm{\fk{k}}_{L^2(\positiveR, \C^2)}^2 . 
				\label{eq:spherical harmonics decomposition norm 2D} \tag{\ref{item:spherical harmonics decomposition 2D}.ii}
			\end{gather}
			Conversely, for any $\{ \fk{k} \} \subset L^2(\R_{>0}, \C^2)$ satisfying
			\begin{equation}
				\sum_{k=-\infty}^{\infty} \norm{\fk{k}}_{L^2(\positiveR, \C^2)}^2 < \infty ,
			\end{equation}
			the function $f$ given by \eqref{eq:spherical harmonics decomposition 2D} is in $L^2(\R^2, \C^2)$ and \eqref{eq:spherical harmonics decomposition norm 2D} holds.
			\eqitem \label{item:Sf norm decomposition Dirac 2D}
			For simplicity, we define $\lambda_k(r)$ for $k \leq -1$ by $\lambda_k(r) \coloneqq \lambda_{\abs{k}}(r)$. 
			Let $f \in L^2(\R^2, \C^2)$ and decompose $\fpm$ as
			\begin{equation}
				\fpm(\xi) =  r^{-1/2} \sum_{k=-\infty}^{\infty} \matrixEk{k}(\theta) \fkpm{k}(r) .
			\end{equation}
			Then we have
			\begin{align}
				&\quad \norm{\DiracS f}_{L^2(\R^{3}, \C^2)}^2 \\
				&= 2 \pi \sum_{k = -\infty}^{\infty} \int_{0}^{\infty} 
				\mleft(
				\Innerproduct{%
					\Lambdak{k}(r)
					\fkp{k}(r)%
				}%
				{\fkp{k}(r)} 
				+ \Innerproduct{%
					\Lambdak{k}(r)
					\fkm{k}(r)%
				}%
				{\fkm{k}(r)} 
				\mright) \, dr ,
			\end{align}
			where
			\begin{equation}
				\Lambdak{k}(r) \coloneqq \begin{pmatrix}
					\lambda_k(r) & 0 \\
					0 & \lambda_{k+1}(r)
				\end{pmatrix} .
			\end{equation}
			\eqitem \label{item:identity 2D} 
			We have
			\begin{equation}
				(\sigma_1 \cos{\theta} + \sigma_2 \sin{\theta}) \matrixEk{k}(\theta)
				= 
				\matrixEk{k}(\theta) \sigma_1 .
			\end{equation}
		\end{eqenumerate}
	\end{lemma}
	\begin{proof}[Proof of Lemma \ref{lem:lemma 2D}]
		\eqref{item:spherical harmonics decomposition 2D} and \eqref{item:Sf norm decomposition Dirac 2D} are immediate from Theorem \ref{thm:spherical harmonics decomposition} and Lemma \ref{lem:Sf norm decomposition Dirac}, respectively. 
		\eqref{item:identity 2D} is also quite easy:
		\begin{align}
			&(\sigma_1 \cos{\theta} + \sigma_2 \sin{\theta}) \matrixEk{k}(\theta) \\
			&= \frac{1}{\sqrt{2\pi}}
			\mleft( \cos{\theta}
			\begin{pmatrix}
				0 & 1 \\
				1 & 0
			\end{pmatrix}
			+ 
			\sin{\theta} 
			\begin{pmatrix}
				0 & -i \\
				i & 0
			\end{pmatrix} 
			\mright)
			\begin{pmatrix}
				e^{i k \theta} & 0 \\ 
				0 & e^{i (k+1) \theta}
			\end{pmatrix}
			\\
			&= \frac{1}{\sqrt{2\pi}}
			\begin{pmatrix}
				0 & e^{- i \theta} \\
				e^{i\theta} & 0
			\end{pmatrix} 
			\begin{pmatrix}
				e^{i k \theta} & 0 \\ 
				0 & e^{i (k+1) \theta}
			\end{pmatrix}
			\\
			&= \frac{1}{\sqrt{2\pi}}
			\begin{pmatrix}
				0 & e^{i k \theta} \\ 
				e^{i (k+1) \theta} & 0
			\end{pmatrix} \\
			&= 
			\matrixEk{k}(\theta) \sigma_1 . \qedhere
		\end{align} 
	\end{proof}
	Using Lemma \ref{lem:lemma 2D}, we prove the following result:
	\begin{theorem}[{\cite[Theorem 2.2]{Iko2022}}] \label{thm:2D Dirac}
		We have
		\begin{equation}
			\norm{\DiracS}_{L^2(\R^2, \C^2) \to L^2(\R^3, \C^2)}^2 = 2 \pi \Diraclambdasup ,
		\end{equation}
		where
		\begin{gather}
			\Diraclambdasup
			\coloneqq \sup_{k \in \Z} \sup_{r > 0} \Diraclambdak{k}(r), \\
			\Diraclambdak{k}(r) \coloneqq 
			\frac{1}{2} ( \lambda_k(r) + \lambda_{k+1}(r) ) + \frac{m}{2\phi_{m}(r)} \abs{ \lambda_k(r) - \lambda_{k+1}(r) } .
		\end{gather}
		Regarding extremisers, let $f \in L^2(\R^2, \C^2)$ be such that
		\begin{equation}
			f(\xi) = r^{-1/2} \sum_{k=-\infty}^{\infty} \matrixEk{k}(\theta) \fk{k}(r) , \quad \xi = (r \cos \theta, r \sin \theta) .
		\end{equation}
		Then the following are equivalent:
		\begin{eqenumerate}
			\eqitem \label{item:extremiser 2D Dirac}
			The equality 
			\begin{equation}
				\norm{\DiracS f}_{L^2(\R^3, \C^2)}^2 = 2 \pi \Diraclambdasup \norm{f}_{L^2(\R^2, \C^2)}^2 
			\end{equation}
			holds.
			\eqitem \label{item:extremiser condition 2D Dirac}
			The functions $\{ \fk{k} \}_{k \in \Z}$ satisfy
			\begin{gather}
				\supp{ \fk{k} } \subset \DiracL_k \coloneqq \set{ r > 0 }{ \Diraclambdak{k}(r) = \Diraclambdasup } , \\
				\fk{k}(r)
				\in W_k(r) , \quad \text{ a.e. } r > 0 
			\end{gather}
			for each $k \in \Z$, 
			where
			\begin{align}
				&W_k(r) \\
				&= 
				\begin{dcases}
					\C^2, & m ( \lambda_k(r) - \lambda_{k+1}(r) ) = 0 , \\
					\text{the eigenspace of $m \sigma_3 + r \sigma_1$ associated with $\phi_{m}(r)$} , 
					& m ( \lambda_k(r) - \lambda_{k+1}(r) )  > 0 , \\
					\text{the eigenspace of $m \sigma_3 + r \sigma_1$ associated with $- \phi_{m}(r)$} ,
					& m ( \lambda_k(r) - \lambda_{k+1}(r) )  < 0 ,
				\end{dcases}
			\end{align}
		\end{eqenumerate}
		As a consequence, extremisers exist if and only if there exists 
		$k \in \Z$ such that $\measure{\DiracL_k} > 0$.
	\end{theorem}
	\begin{proof}[Proof of Theorem \ref{thm:2D Dirac}]
		Let $f \in L^2(\R^2, \C^2)$ be such that
		\begin{gather}
			f(\xi) = r^{-1/2} \sum_{k=-\infty}^{\infty} \matrixEk{k}(\theta) \fk{k}(r) , \quad \xi = (r \cos \theta, r \sin \theta) , \\
			\norm{f}_{L^2(\R^2, \C^2)}^2 = \sum_{k=-\infty}^{\infty}  \norm{\fk{k}}_{L^2(\positiveR, \C^2)}^2 .
		\end{gather}
		At first we need to compute $\fpm$.
		By \eqref{item:identity 2D} and $\sigma_3 \matrixEk{k} = \matrixEk{k} \sigma_3$, we have
		\begin{align}
			\fpm(\xi)
			&= 
			\frac{1}{2 \abs{\xi}^{1/2}} \mleft( f(\xi) \pm \frac{1}{\phi_{m}(\abs{\xi})} \bigg( m \sigma_3 f(\xi) + \sum_{j=1}^{2} \sigma_j \xi_j f(\xi) \bigg) \mright) \\
			&= 
			\frac{1}{2r^{1/2}} \sum_{k=-\infty}^{\infty}\mleft( \matrixEk{k}(\theta) \pm \frac{1}{\phi_{m}(r)} ( m\sigma_3 \matrixEk{k}(\theta) + r \matrixEk{k}(\theta) \sigma_1 ) \mright) \fk{k}(r) \\
			&= 
			\frac{1}{2r^{1/2}} \sum_{k=-\infty}^{\infty}  \matrixEk{k}(\theta) \mleft( \matrixI \pm \frac{1}{\phi_{m}(r)} ( m\sigma_3 + r \sigma_1 ) \mright) \fk{k}(r) .
		\end{align}
		Therefore, \eqref{item:Sf norm decomposition Dirac 2D} implies
		\begin{equation}
			\norm{\DiracS f}_{L^2(\R^{3}, \C^2)}^2 \\ 
			= 2 \pi \sum_{k=-\infty}^{\infty} \int_{0}^{\infty} 
			\innerproduct{
				\DiracLambdak{k}(r) \fk{k}(r)
			}
			{
				\fk{k}(r)
			}  \, dr ,
		\end{equation}
		where
		\begin{align}
			2 \DiracLambdak{k}(r)&\coloneqq \frac{1}{2} \mleft( \matrixI + \frac{1}{\phi_{m}} ( m\sigma_3 + r \sigma_1 ) \mright) \Lambdak{k} \mleft( \matrixI + \frac{1}{\phi_{m}} ( m\sigma_3 + r \sigma_1 ) \mright) \\
			&\quad+ \frac{1}{2} \mleft( \matrixI - \frac{1}{\phi_{m}} ( m\sigma_3 + r \sigma_1 ) \mright) \Lambdak{k} \mleft( \matrixI - \frac{1}{\phi_{m}} ( m\sigma_3 + r \sigma_1 ) \mright)\\
			&=  \Lambdak{k} + \frac{1}{\phi_{m}^2} ( m\sigma_3 + r \sigma_1 ) \Lambdak{k} ( m\sigma_3 + r \sigma_1 ) \\
			&=  \Lambdak{k}  + \frac{1}{\phi_{m}^2}
			(m^2 \sigma_3 \Lambdak{k} \sigma_3 + mr ( \sigma_3 \Lambdak{k} \sigma_1 + \sigma_1 \Lambdak{k} \sigma_3 ) + r^2 \sigma_1 \Lambdak{k} \sigma_1  ) \\
			&=  \Lambdak{k}  + \frac{1}{\phi_{m}^2}
			(m^2 \Lambdak{k} + mr (\lambda_k - \lambda_{k+1}) \sigma_1 + (\phi_{m}^2 - m^2) \sigma_1 \Lambdak{k} \sigma_1  ) \\ 
			&=  \Lambdak{k} + \sigma_1 \Lambdak{k} \sigma_1 + \frac{m}{\phi_{m}^2}
			(m ( \Lambdak{k} - \sigma_1 \Lambdak{k} \sigma_1 ) + r (\lambda_k - \lambda_{k+1}) \sigma_1  ) \\ 
			&= (\lambda_k(r) + \lambda_{k+1}(r)) \matrixI + \frac{m}{\phi_{m}(r)^2} (\lambda_k(r) - \lambda_{k+1}(r)) (m \sigma_3 + r \sigma_1) .
		\end{align}
		Now we need to determine the maximal eigenvalue of $\DiracLambdak{k}(r)$ and its associated eigenspace. 
		Since eigenvalues of the matrix $m \sigma_3 + r \sigma_1$ are $\pm \phi_{m}(r)$,
		we conclude that the maximal eigenvalue of $\DiracLambdak{k}(r)$ and its associated eigenspace are
		\begin{equation}
			\Diraclambdak{k}(r) 
			= \frac{1}{2} ( \lambda_k(r) + \lambda_{k+1}(r) ) + \frac{m}{2\phi_{m}(r)} \abs{ \lambda_k(r) - \lambda_{k+1}(r) }
		\end{equation}
		and
		\begin{align}
			&W_k(r) \\
			&= 
			\begin{dcases}
				\C^2, & m ( \lambda_k(r) - \lambda_{k+1}(r) ) = 0 , \\
				\text{the eigenspace of $m \sigma_3 + r \sigma_1$ associated with $\phi_{m}(r)$} , 
				& m ( \lambda_k(r) - \lambda_{k+1}(r) )  > 0 , \\
				\text{the eigenspace of $m \sigma_3 + r \sigma_1$ associated with $- \phi_{m}(r)$} ,
				& m ( \lambda_k(r) - \lambda_{k+1}(r) )  < 0 ,
			\end{dcases}
		\end{align}
		respectively.
		Therefore, we have
		\begin{align}
			\norm{\DiracS f}_{L^2(\R^{3}, \C^2)}^2 
			&= 2 \pi \sum_{k=-\infty}^{\infty} \int_{0}^{\infty} 
			\innerproduct{
				\DiracLambdak{k}(r) \fk{k}(r)
			}
			{
				\fk{k}(r)
			}  \, dr \\
			&\leq 2 \pi \sum_{k=-\infty}^{\infty} \int_{0}^{\infty} \Diraclambdak{k}(r)
			\abs{\fk{k}(r)}^2 \, dr \\
			&\leq 2 \pi \Diraclambdasup \norm{f}_{L^2(\R^2, \C^2)}^2 
		\end{align}
		and hence
		\begin{equation}
			\norm{\DiracS}^2 \leq 2 \pi \Diraclambdasup .
		\end{equation}
		
		To see the equality $\norm{\DiracS}^2 = 2 \pi \Diraclambdasup$, 
		we will show that for any $\varepsilon > 0$, 
		there exists $f \in L^2(\R^2, \C^2) \setminus \{0\}$ 
		such that $\norm{\DiracS f}^{2}_{L^2(\R^3, \C^2)} \geq 2 \pi ( \Diraclambdasup - \varepsilon ) \norm{f}_{L^2(\R^2, \C^2)}^2$.
		Fix $\varepsilon > 0$. 
		Then, by the definition of $\Diraclambdasup$ and the continuity of $\Diraclambdak{k}$, 
		we can choose $k \in \Z$ such that the Lebesgue measure of
		\begin{equation}
			\DiracL_k(\varepsilon) \coloneqq \set{r > 0}{ \Diraclambda_k(r) \geq \Diraclambdasup - \varepsilon }	
		\end{equation}
		is nonzero (possibly infinite). 
		Now let $f \in L^2(\R^2, \C^2) \setminus \{ 0 \}$ be
		\begin{equation}
			f(\xi) = \matrixEk{k}(\theta) \fk{k}(r)
		\end{equation}
		with
		$\fk{k} \in L^2(\positiveR, \C^2)$ satisfying
		\begin{gather}
			\supp{ \fk{k} } \subset \DiracL_k (\varepsilon) , \\
			\fk{k}(r)
			\in W_k(r) , \quad \text{ a.e. } r > 0 .
		\end{gather}
		Then we have
		\begin{align}
			\norm{\DiracS f}_{L^2(\R^{3}, \C^2)}^2 
			&= 2 \pi  \int_{0}^{\infty}
			\innerproduct{
				\DiracLambdak{k}(r) \fk{k}(r)
			}
			{
				\fk{k}(r)
			}  \, dr \\
			&= 2 \pi  \int_{0}^{\infty} \Diraclambdak{k}(r)
			\abs{\fk{k}(r)}^2 \, dr \\
			&\geq 2 \pi (\Diraclambdasup - \varepsilon) \norm{f}_{L^2(\R^2, \C^2)}^2 ,
		\end{align}
		hence the equality $\norm{\DiracS}^2 = 2 \pi \Diraclambdasup$ holds.
		
		Using 
		a similar argument, we can show that $\eqref{item:extremiser condition 2D Dirac} \implies \eqref{item:extremiser 2D Dirac}$.
		Suppose that $\{ \fk{k} \}_{k \in \Z}$ satisfies
		\begin{gather}
			\supp{ \fk{k} } \subset \DiracL_k , \\
			\fk{k}(r)
			\in W_k(r) , \quad \text{ a.e. } r > 0 
		\end{gather}
		for each $k \in \Z$.
		Then we have
		\begin{align}
			\norm{\DiracS f}_{L^2(\R^{3}, \C^2)}^2 
			&= 2 \pi  \int_{0}^{\infty} \sum_{k=-\infty}^{\infty}
			\innerproduct{
				\DiracLambdak{k}(r) \fk{k}(r)
			}
			{
				\fk{k}(r)
			}  \, dr \\
			&= 2 \pi  \int_{0}^{\infty} \sum_{k=-\infty}^{\infty} \Diraclambdak{k}(r)
			\abs{\fk{k}(r)}^2 \, dr \\
			&= 2 \pi \Diraclambdasup \norm{f}_{L^2(\R^2, \C^2)}^2 .
		\end{align}
		
		Finally, we show that $\eqref{item:extremiser 2D Dirac} \implies \eqref{item:extremiser condition 2D Dirac}$.
		Suppose that the equality $\norm{\DiracS f}_{L^2(\R^3, \C^2)}^2 = 2 \pi \Diraclambdasup \norm{f}_{L^2(\R^2, \C^2)}^2$ holds. 
		Then, since
		\begin{align}
			&\quad 2\pi \sum_{k \in \Z} \int_{0}^{\infty} \mleft( \Diraclambdasup \abs{ \fk{k}(r) }^2 - \innerproduct{
				\DiracLambdak{k}(r) \fk{k}(r)
			}
			{
				\fk{k}(r)
			} \mright)  \, dr
			\\
			&= 2 \pi \Diraclambdasup \norm{f}_{L^2(\R^2, \C^2)}^2 - \norm{\DiracS f}_{L^2(\R^3, \C^2)}^2 \\
			&= 0
		\end{align}
		and 
		\begin{equation}
			\Diraclambdasup \abs{ \fk{k}(r) }^2 - \innerproduct{
				\DiracLambdak{k}(r) \fk{k}(r)
			}
			{
				\fk{k}(r)
			} \geq 0 ,
		\end{equation}
		we obtain 
		\begin{gather}
			\fk{k}(r) \in W_k(r) , \\
			\Diraclambdasup \abs{ \fk{k}(r) }^2 - \innerproduct{
				\DiracLambdak{k}(r) \fk{k}(r)
			}
			{
				\fk{k}(r)
			}
			= ( \Diraclambdasup - \Diraclambda_k(r) ) \abs{ \fk{k}(r) }^2  = 0 
		\end{gather}
		for almost every $r > 0$ and each $k \in \Z$.
		On the other hand, the definition of $\DiracL_k$ implies that 
		\begin{equation}
			\Diraclambdasup - \Diraclambda_k(r) > 0 
		\end{equation}
		for any $r \in \positiveR \setminus \DiracL_k$ and each $k \in \Z$.
		Therefore, we conclude that 
		\begin{equation}
			\fk{k}(r) = 0 
		\end{equation}
		holds for almost every $r \in \positiveR \setminus \DiracL_k$ and each $k \in \Z$.
	\end{proof}
	\section{In the case \texorpdfstring{$d=3$}{d=3}} \label{section:proof for 3D}
	In the case $d = 3$, it is known that the following functions $\{ \Ykn{k}{n} \}_{\abs{n} \leq k}$ form an orthonormal basis for $\HP_{k}(\R^3)$ for each $k \geq 0$:
	\begin{equation} \label{eq:spherical harmonics}
		\Ykn{k}{n}(\theta, \varphi) = \frac{1}{\sqrt{2\pi}} (-1)^{(n + \abs{n})/2} \Normalizekn{k}{\abs{n}} ( \sin \theta )^{\abs{n}} \GPpn{\abs{n} + 1/2}{k - \abs{n}} (\cos \theta) e^{i n \varphi}, 
	\end{equation}
	where $\GPpn{p}{n}$ denotes the Gegenbauer polynomial, which is defined by the following recurrence relation:
	\begin{equation}
		\begin{cases}
			\GPpn{p}{-1}(x) = 0 , \\
			\GPpn{p}{0}(x) = 1 , \\
			(n+1) \GPpn{p}{n+1}(x) = 2 (n + p) x \GPpn{p}{n}(x) - ( n + 2p - 1) \GPpn{p}{n-1}(x) , 
		\end{cases}
	\end{equation}
	and $\Normalizekn{k}{\abs{n}}$ is the normalizing constant given by
	\begin{equation} \label{eq:normalizing}
		\Normalizekn{k}{n} \coloneqq (2n-1)!!
		\mleft( (k + 1/2) \frac{  (k - n)! }{ (k + n)! } \mright)^{1/2} .
	\end{equation}
	At first we prove the following lemma:
	\begin{lemma} \label{lem:identity for 3D}
		Let 
		\begin{equation}
			\matrixYkn{k}{n}(\theta, \varphi) \\
			= \begin{pmatrix}
				\Ykn{k}{n}(\theta, \varphi) & 0\\
				0 & \Ykn{k}{n+1}(\theta, \varphi)
			\end{pmatrix} . 
		\end{equation}
		Then there exist matrices $\{ \Akn{k}{n} \} \subset M_{2 \times 2}(\R)$ satisfying the following properties:
		\begin{gather}
			\mleft( \sigma_1 \sin{\theta} \cos{\varphi} 
			+ 
			\sigma_2 \sin{\theta} \sin{\varphi} 
			+ \sigma_3 \cos{\theta}
			\mright)
			\matrixYkn{k}{n}(\theta, \varphi)
			= 
			\matrixYkn{k+1}{n}(\theta, \varphi) \Akn{k}{n}
			+
			\matrixYkn{k-1}{n}(\theta, \varphi) \transposeAkn{k-1}{n}, 
			\label{eq:lemma 1} \\
			\Akn{k+1}{n} \Akn{k}{n} = \matrixO , 
			\label{eq:lemma 2}\\
			\transposeAkn{k}{n} \Akn{k}{n} + \Akn{k-1}{n} \transposeAkn{k-1}{n} = \matrixI ,
			\label{eq:lemma 3}\\
			\det\Akn{k}{n} = 0.
			\label{eq:lemma 4}
		\end{gather}
		
	\end{lemma}
	\begin{proof}[Proof of Lemma \ref{lem:identity for 3D}]	
		At first we prove \eqref{eq:lemma 1} for $n \geq 0$. By \eqref{eq:spherical harmonics}, we have
		\begin{align}
			&\quad \matrixYkn{k}{n}(\theta, \varphi) \\
			&= \begin{pmatrix}
				\Ykn{k}{n}(\theta, \varphi) & 0\\
				0 & \Ykn{k}{n+1}(\theta, \varphi)
			\end{pmatrix} \\
			&= \matrixEk{n}(\varphi) \begin{pmatrix}
				(-1)^{n} \Normalizekn{k}{n} ( \sin \theta )^{n} \GPpn{n + 1/2}{k - n} (\cos \theta)  & 0\\
				0 & (-1)^{n+1} \Normalizekn{k}{n+1} ( \sin \theta )^{n+1} \GPpn{n + 3/2}{k - n - 1} (\cos \theta) 
			\end{pmatrix} \\
			&\eqqcolon \matrixEk{n}(\varphi) \matrixCkn{k}{n}(\theta) , 
		\end{align}
		hence Lemma \ref{lem:lemma 2D} implies
		\begin{align}
			& \mleft( \sigma_1 \sin{\theta} \cos{\varphi}
			+ 
			\sigma_2 \sin{\theta} \sin{\varphi}
			\mright)
			\matrixYkn{k}{n} (\theta, \varphi)
			\\
			&= \mleft( \sigma_1 \cos{\varphi} 
			+ 
			\sigma_2 \sin{\varphi} 
			\mright)
			\matrixEk{n}(\varphi) \matrixCkn{k}{n}(\theta) \sin{\theta}  \\
			&= \matrixEk{n}(\varphi) \sigma_1 \matrixCkn{k}{n}(\theta) \sin{\theta} .
		\end{align}
		Since
		\begin{align}
			\sigma_3 \matrixYkn{k}{n}(\theta, \varphi) \cos{\theta} = \matrixEk{n}(\varphi) \matrixCkn{k}{n}(\theta) \sigma_3 \cos{\theta} ,
		\end{align}
		we obtain 
		\begin{equation}
			\mleft( \sigma_1 \sin{\theta} \cos{\varphi} 
			+ 
			\sigma_2 \sin{\theta} \sin{\varphi} 
			+ \sigma_3 \cos{\theta}
			\mright)
			\matrixYkn{k}{n}(\theta, \varphi)
			= 
			\matrixEk{n}(\varphi) ( \sigma_1 \matrixCkn{k}{n}(\theta) \sin{\theta} + \matrixCkn{k}{n}(\theta) \sigma_3 \cos{\theta} ) .
		\end{equation}
		Therefore, it is enough to show that
		\begin{equation}
			\sigma_1 \matrixCkn{k}{n}(\theta) \sin{\theta} + \matrixCkn{k}{n}(\theta) \sigma_3 \cos{\theta} = 
			\matrixCkn{k+1}{n}(\theta) \Akn{k}{n}
			+
			\matrixCkn{k-1}{n}(\theta) \transposeAkn{k-1}{n} .
		\end{equation}
		
		In order to prove this, 
		we use the following identities of the Gegenbauer polynomials:
		\begin{gather}
			(n + p) \GPpn{p}{n}(x) = p ( \GPpn{p+1}{n}(x) - \GPpn{p+1}{n-2}(x) ) ,
			\label{eq:Gegenbauer identity 1} \\
			4p (n + p) (1 - x^2) \GPpn{p+1}{n-1}(x) = (n + 2p - 1)(n + 2p) \GPpn{p}{n-1}(x) - n (n + 1) \GPpn{p}{n + 1}(x) , 
			\label{eq:Gegenbauer identity 2} \\
			2 (n + p) x \GPpn{p}{n}(x) = (n+1) \GPpn{p}{n+1}(x) + ( n + 2p - 1) \GPpn{p}{n-1}(x) .
			\label{eq:Gegenbauer identity 3}
		\end{gather}
		Using \eqref{eq:Gegenbauer identity 1} and \eqref{eq:Gegenbauer identity 2}, we obtain
		\begin{align}
			&\quad (2k + 1) (-1)^{n} \Normalizekn{k}{n} ( \sin \theta )^{n+1} \GPpn{n + 1/2}{k - n} (\cos \theta) \\
			&= (-1)^{n} \Normalizekn{k}{n} ( \sin \theta )^{n+1} (2n + 1) ( \GPpn{n+3/2}{(k+1)-(n+1)}(\cos \theta) - \GPpn{n+3/2}{(k-1) - (n+1)}(\cos \theta) ) \\
			&= (-1)^{n} ( \sin \theta )^{n+1} (2n + 1) \mleft(  \Normalizekn{k}{n}  \GPpn{n+3/2}{(k+1)-(n+1)}(\cos \theta) -\Normalizekn{k}{n}  \GPpn{n+3/2}{(k-1) - (n+1)}(\cos \theta) \mright) 
		\end{align}
		and
		\begin{align}
			&\quad(2n + 1) (2k + 1) (-1)^{n+1} \Normalizekn{k}{n+1} ( \sin \theta )^{n + 2} \GPpn{n+3/2}{k - (n+1)} (\cos \theta) \\
			&= (-1)^{n+1} (2n + 1) (2k + 1) ( \sin \theta )^{n} \Normalizekn{k}{n+1} ( 1 - (\cos{\theta})^2 ) \GPpn{n+3/2}{k - (n+1)} (\cos \theta) \\
			&= (-1)^{n+1} ( \sin \theta )^{n} \Normalizekn{k}{n+1} ( (k + n )(k + n + 1) \GPpn{n + 1/2}{(k-1) - n} (\cos \theta) - (k - n)(k - n + 1) \GPpn{n+1/2}{(k+1) - n} (\cos \theta) ) \\
			&= (-1)^{n+1} ( \sin \theta )^{n} \mleft( (k + n )(k + n + 1) \Normalizekn{k}{n+1} \GPpn{n + 1/2}{(k-1) - n} (\cos \theta) - (k - n)(k - n + 1) \Normalizekn{k}{n+1} \GPpn{n+1/2}{(k+1) - n} (\cos \theta) \mright) ,
		\end{align}
		respectively. Combining these, we get 
		\begin{align} 
			& \matrixCkn{k}{n}(\theta) \sin{\theta} \\
			&= (-1)^{n} ( \sin \theta )^{n}
			\begin{pmatrix}
				- \Normalizekn{k+1}{n+1} ( \sin \theta )  \GPpn{n+3/2}{(k+1)-(n+1)}(\cos \theta)  & 0 \\
				0 &   \Normalizekn{k+1}{n}  \GPpn{n+1/2}{(k+1) - n} (\cos \theta)
			\end{pmatrix}
			\begin{pmatrix}
				- \frac{ (2n+1) \Normalizekn{k}{n} }{ (2k+1 )\Normalizekn{k+1}{n+1} } & 0 \\
				0 & \frac{ (k - n )(k - n + 1) \Normalizekn{k}{n+1} }{ (2n + 1) (2k + 1) \Normalizekn{k+1}{n} } 
			\end{pmatrix} \\
			&+ (-1)^{n} ( \sin \theta )^{n}
			\begin{pmatrix}
				- \Normalizekn{k-1}{n+1} ( \sin \theta ) \GPpn{n+3/2}{(k-1) - (n+1)}(\cos \theta)  & 0 \\
				0 &  \Normalizekn{k-1}{n} \GPpn{n + 1/2}{(k-1) - n} (\cos \theta)
			\end{pmatrix}
			\begin{pmatrix}
				\frac{ (2n+1) \Normalizekn{k}{n} }{ (2k+1 )\Normalizekn{k-1}{n+1} } & 0 \\
				0 & - \frac{ (k + n )(k + n + 1) \Normalizekn{k}{n+1} }{ (2n + 1) (2k + 1) \Normalizekn{k-1}{n} }
			\end{pmatrix} \\
			&= \sigma_1 \matrixCkn{k+1}{n}(\theta) \sigma_1
			\begin{pmatrix}
				- \frac{ (2n+1) \Normalizekn{k}{n} }{ (2k+1 )\Normalizekn{k+1}{n+1} } & 0 \\
				0 & \frac{ (k - n )(k - n + 1) \Normalizekn{k}{n+1} }{ (2n + 1) (2k + 1) \Normalizekn{k+1}{n} } 
			\end{pmatrix}
			+ \sigma_1 \matrixCkn{k-1}{n}(\theta) \sigma_1
			\begin{pmatrix}
				\frac{ (2n+1) \Normalizekn{k}{n} }{ (2k+1 )\Normalizekn{k-1}{n+1} } & 0 \\
				0 & - \frac{ (k + n )(k + n + 1) \Normalizekn{k}{n+1} }{ (2n + 1) (2k + 1) \Normalizekn{k-1}{n} }
			\end{pmatrix} \\
			&= \sigma_1 \matrixCkn{k+1}{n}(\theta) 
			\begin{pmatrix}
				0 & \frac{ (k - n )(k - n + 1) \Normalizekn{k}{n+1} }{ (2n + 1) (2k + 1) \Normalizekn{k+1}{n} }  \\
				- \frac{ (2n+1) \Normalizekn{k}{n} }{ (2k+1 )\Normalizekn{k+1}{n+1} } & 0
			\end{pmatrix}
			+ \sigma_1 \matrixCkn{k-1}{n}(\theta) 
			\begin{pmatrix}
				0 & - \frac{ (k + n )(k + n + 1) \Normalizekn{k}{n+1} }{ (2n + 1) (2k + 1) \Normalizekn{k-1}{n} } \\
				\frac{ (2n+1) \Normalizekn{k}{n} }{ (2k+1 )\Normalizekn{k-1}{n+1} } & 0
			\end{pmatrix} .
		\end{align}
		Furthermore, \eqref{eq:Gegenbauer identity 3} implies
		\begin{align}
			& (2k + 1) (-1)^{n} \Normalizekn{k}{n} ( \sin \theta )^{n} \GPpn{n + 1/2}{k - n} (\cos \theta) \cos{\theta} \\
			&\quad = (-1)^{n}  ( \sin \theta )^{n} \mleft( (k-n+1) \Normalizekn{k}{n} \GPpn{n+1/2}{(k+1)-n}(\cos \theta) +  (k+n) \Normalizekn{k}{n} \GPpn{n + 1/2}{(k-1)-n}(\theta) \mright)
		\end{align}
		and so that
		\begin{align} 
			& \matrixCkn{k}{n}(\theta) \sigma_3 \cos{\theta} \\
			&= (-1)^{n} ( \sin \theta )^{n}
			\begin{pmatrix}
				\Normalizekn{k+1}{n}  \GPpn{n+1/2}{(k+1)-n}(\cos \theta)  & 0 \\
				0 &  - \Normalizekn{k+1}{n+1} ( \sin \theta ) \GPpn{n+3/2}{(k+1)-(n+1)}(\cos \theta)
			\end{pmatrix}
			\begin{pmatrix}
				\frac{(k-n+1) \Normalizekn{k}{n}}{ (2k+1) \Normalizekn{k+1}{n} } & 0 \\
				0 & \frac{(k-n) \Normalizekn{k}{n+1}}{ (2k+1) \Normalizekn{k+1}{n+1} }
			\end{pmatrix} \sigma_3 \\
			&+ (-1)^{n} ( \sin \theta )^{n}
			\begin{pmatrix}
				\Normalizekn{k-1}{n}  \GPpn{n+1/2}{(k-1)-n}(\cos \theta)  & 0 \\
				0 & - \Normalizekn{k-1}{n+1} ( \sin \theta ) \GPpn{n+3/2}{(k-1)-(n+1)}(\cos \theta)
			\end{pmatrix}
			\begin{pmatrix}
				\frac{(k+n) \Normalizekn{k}{n}}{ (2k+1) \Normalizekn{k-1}{n} } & 0 \\
				0 & \frac{(k+n+1) \Normalizekn{k}{n+1}}{ (2k+1) \Normalizekn{k-1}{n+1} }
			\end{pmatrix} \sigma_3 \\
			&= \matrixCkn{k+1}{n}(\theta)
			\begin{pmatrix}
				\frac{(k-n+1) \Normalizekn{k}{n}}{ (2k+1) \Normalizekn{k+1}{n} } & 0 \\
				0 & - \frac{(k-n) \Normalizekn{k}{n+1}}{ (2k+1) \Normalizekn{k+1}{n+1} }
			\end{pmatrix} 
			+ \matrixCkn{k-1}{n}(\theta)
			\begin{pmatrix}
				\frac{(k+n) \Normalizekn{k}{n}}{ (2k+1) \Normalizekn{k-1}{n} } & 0 \\
				0 & - \frac{(k+n+1) \Normalizekn{k}{n+1}}{ (2k+1) \Normalizekn{k-1}{n+1} } 
			\end{pmatrix} .
		\end{align}
		Therefore, we conclude that
		\begin{align}
			\quad &\sigma_1 \matrixCkn{k}{n}(\theta) \sin{\theta} + \matrixCkn{k}{n}(\theta) \sigma_3 \cos{\theta} \\
			&= \matrixCkn{k+1}{n}(\theta) 
			\begin{pmatrix}
				\frac{(k-n+1) \Normalizekn{k}{n}}{ (2k+1) \Normalizekn{k+1}{n} } & \frac{ (k - n )(k - n + 1) \Normalizekn{k}{n+1} }{ (2n + 1) (2k + 1) \Normalizekn{k+1}{n} }  \\
				- \frac{ (2n+1) \Normalizekn{k}{n} }{ (2k+1 )\Normalizekn{k+1}{n+1} } & - \frac{(k-n) \Normalizekn{k}{n+1}}{ (2k+1) \Normalizekn{k+1}{n+1} }
			\end{pmatrix}
			+  \matrixCkn{k-1}{n}(\theta) 
			\begin{pmatrix}
				\frac{(k+n) \Normalizekn{k}{n}}{ (2k+1) \Normalizekn{k-1}{n} } & - \frac{ (k + n )(k + n + 1) \Normalizekn{k}{n+1} }{ (2n + 1) (2k + 1) \Normalizekn{k-1}{n} } \\
				\frac{ (2n+1) \Normalizekn{k}{n} }{ (2k+1 )\Normalizekn{k-1}{n+1} } & - \frac{(k+n+1) \Normalizekn{k}{n+1}}{ (2k+1) \Normalizekn{k-1}{n+1} } 
			\end{pmatrix} 
		\end{align}
		holds.
		Now let
		\begin{equation}
			\Akn{k}{n} \coloneqq \begin{pmatrix}
				\frac{(k-n+1) \Normalizekn{k}{n}}{ (2k+1) \Normalizekn{k+1}{n} } & \frac{ (k - n )(k - n + 1) \Normalizekn{k}{n+1} }{ (2n + 1) (2k + 1) \Normalizekn{k+1}{n} }  \\
				- \frac{ (2n+1) \Normalizekn{k}{n} }{ (2k+1 )\Normalizekn{k+1}{n+1} } & - \frac{(k-n) \Normalizekn{k}{n+1}}{ (2k+1) \Normalizekn{k+1}{n+1} }
			\end{pmatrix} .
		\end{equation}
		Then we have
		\begin{align}
			&\quad \transposeAkn{k-1}{n} - 	\begin{pmatrix}
				\frac{(k+n) \Normalizekn{k}{n}}{ (2k+1) \Normalizekn{k-1}{n} } & - \frac{ (k + n )(k + n + 1) \Normalizekn{k}{n+1} }{ (2n + 1) (2k + 1) \Normalizekn{k-1}{n} } \\
				\frac{ (2n+1) \Normalizekn{k}{n} }{ (2k+1 )\Normalizekn{k-1}{n+1} } & - \frac{(k+n+1) \Normalizekn{k}{n+1}}{ (2k+1) \Normalizekn{k-1}{n+1} } 
			\end{pmatrix} \\
			&=   \begin{pmatrix}
				\frac{(k-n) \Normalizekn{k-1}{n}}{ (2k-1) \Normalizekn{k}{n} } & - \frac{ (2n+1) \Normalizekn{k-1}{n} }{ (2k-1 )\Normalizekn{k}{n+1} }  \\
				\frac{ (k - n - 1)(k - n) \Normalizekn{k-1}{n+1} }{ (2n + 1) (2k - 1) \Normalizekn{k}{n} }  & - \frac{(k-n-1) \Normalizekn{k-1}{n+1}}{ (2k-1) \Normalizekn{k}{n+1} }
			\end{pmatrix} - \begin{pmatrix}
				\frac{(k+n) \Normalizekn{k}{n}}{ (2k+1) \Normalizekn{k-1}{n} } & - \frac{ (k + n )(k + n + 1) \Normalizekn{k}{n+1} }{ (2n + 1) (2k + 1) \Normalizekn{k-1}{n} } \\
				\frac{ (2n+1) \Normalizekn{k}{n} }{ (2k+1 )\Normalizekn{k-1}{n+1} } & - \frac{(k+n+1) \Normalizekn{k}{n+1}}{ (2k+1) \Normalizekn{k-1}{n+1} } 
			\end{pmatrix} \\
			&= \begin{pmatrix}
				\frac{(k-n) \Normalizekn{k-1}{n}}{ (2k-1) \Normalizekn{k}{n} } - \frac{(k+n) \Normalizekn{k}{n}}{ (2k+1) \Normalizekn{k-1}{n} } & \frac{ (k + n )(k + n + 1) \Normalizekn{k}{n+1} }{ (2n + 1) (2k + 1) \Normalizekn{k-1}{n} } - \frac{ (2n+1) \Normalizekn{k-1}{n} }{ (2k-1 )\Normalizekn{k}{n+1} }  \\
				\frac{ (k - n - 1)(k - n) \Normalizekn{k-1}{n+1} }{ (2n + 1) (2k - 1) \Normalizekn{k}{n} } - \frac{ (2n+1) \Normalizekn{k}{n} }{ (2k+1 )\Normalizekn{k-1}{n+1} }  & \frac{(k+n+1) \Normalizekn{k}{n+1}}{ (2k+1) \Normalizekn{k-1}{n+1} }  - \frac{(k-n-1) \Normalizekn{k-1}{n+1}}{ (2k-1) \Normalizekn{k}{n+1} }
			\end{pmatrix} \\
			&= \begin{pmatrix}
				\frac{(2k+1) (k-n) ( \Normalizekn{k-1}{n} )^2 - (2k-1) (k+n) ( \Normalizekn{k}{n} )^2 }{ (2k-1) (2k+1) \Normalizekn{k}{n} \Normalizekn{k-1}{n} }& \frac{ (2k-1)(k + n )(k + n + 1) ( \Normalizekn{k}{n+1} )^2 - (2n+1)^2 (2k + 1) ( \Normalizekn{k-1}{n} )^2 }{ (2n + 1) (2k + 1)  (2k-1 )\Normalizekn{k}{n+1} \Normalizekn{k-1}{n} } \\
				\frac{ (2k+1)(k - n )(k - n - 1) ( \Normalizekn{k-1}{n+1} )^2 - (2n+1)^2 (2k - 1) ( \Normalizekn{k}{n} )^2 }{ (2n + 1) (2k + 1)  (2k-1 )\Normalizekn{k}{n} \Normalizekn{k-1}{n+1} }  & \frac{(2k-1) (k+n+1) ( \Normalizekn{k}{n+1} )^2 - (2k+1) (k-n-1) ( \Normalizekn{k-1}{n+1} )^2 }{ (2k-1) (2k+1) \Normalizekn{k}{n+1} \Normalizekn{k-1}{n+1} } 
			\end{pmatrix} ,
		\end{align}
		and substituting the explicit expression for the normalizing constant \eqref{eq:normalizing} shows that
		\begin{align}
			&\quad (2k+1) (k-n) ( \Normalizekn{k-1}{n} )^2 - (2k-1) (k+n) ( \Normalizekn{k}{n} )^2 \\
			&= (2k+1) (k-n) ( (2n-1)!! )^2
			(k - 1/2) \frac{  (k - n - 1)! }{ (k + n - 1)! } - (2k-1) (k+n) ( (2n-1)!! )^2 (k + 1/2) \frac{  (k - n)! }{ (k + n)! } \\
			&= 0, 
		\end{align}
		that 
		\begin{align}
			&\quad (2k-1)(k + n )(k + n + 1) ( \Normalizekn{k}{n+1} )^2 - (2n+1)^2 (2k + 1) ( \Normalizekn{k-1}{n} )^2 \\
			&= (2k-1) (k+n) (k+n+1) ( (2n+1)!! )^2 (k + 1/2) \frac{  (k - n - 1)! }{ (k + n + 1)! } - (2n+1)^2 (2k + 1) ( (2n-1)!! )^2	(k - 1/2) \frac{  (k - n - 1)! }{ (k + n - 1)! } \\
			&= 0, 
		\end{align}
		that 
		\begin{align}
			&\quad (2k+1)(k - n )(k - n - 1) ( \Normalizekn{k-1}{n+1} )^2 - (2n+1)^2 (2k- 1) ( \Normalizekn{k}{n} )^2 \\
			&= (2k+1) (k-n) (k-n-1) ( (2n+1)!! )^2 (k - 1/2) \frac{  (k - n - 2)! }{ (k + n)! } - (2n+1)^2 (2k - 1) ( (2n-1)!! )^2 (k + 1/2) \frac{  (k - n)! }{ (k + n)! } \\
			&= 0, 
		\end{align}
		and that 
		\begin{align}
			&\quad (2k-1) (k+n+1) ( \Normalizekn{k}{n+1} )^2 - (2k+1) (k-n-1) ( \Normalizekn{k-1}{n+1} )^2 \\
			&= (2k-1) (k+n+1) ( (2n+1)!! )^2 (k + 1/2) \frac{  (k - n - 1)! }{ (k + n + 1)! } - (2k+1) (k-n-1) ( (2n+1)!! )^2 (k - 1/2) \frac{  (k - n - 2)! }{ (k + n)! } \\
			&= 0 .
		\end{align}
		Consequently, we obtain 
		\begin{equation}
			\mleft( \sigma_1 \sin{\theta} \cos{\varphi} 
			+ 
			\sigma_2 \sin{\theta} \sin{\varphi} 
			+ \sigma_3 \cos{\theta}
			\mright)
			\matrixYkn{k}{n}(\theta, \varphi)
			= 
			\matrixYkn{k+1}{n}(\theta, \varphi) \Akn{k}{n}
			+
			\matrixYkn{k-1}{n}(\theta, \varphi) \transposeAkn{k-1}{n}
		\end{equation}
		for $n \geq 0$. 
		
		Now we consider the case $n \leq -1$. 
		In this case, $\Ykn{k}{n}(\theta, \varphi) = \Ykn{k}{-n} (-\theta, -\varphi) $ implies $\matrixYkn{k}{n}(\theta, \varphi) = \sigma_1 \matrixYkn{k}{-(n+1)} (-\theta, -\varphi) \sigma_1$ and so that 
		\begin{align}
			&\quad \mleft( \sigma_1 \sin{\theta} \cos{\varphi} 
			+ 
			\sigma_2 \sin{\theta} \sin{\varphi} 
			+ \sigma_3 \cos{\theta}
			\mright)
			\matrixYkn{k}{n}(\theta, \varphi) \\
			&= \mleft( \sigma_1 \sin{\theta} \cos{\varphi} 
			+ 
			\sigma_2 \sin{\theta} \sin{\varphi} 
			+ \sigma_3 \cos{\theta}
			\mright) \sigma_1 \matrixYkn{k}{-(n+1)} (-\theta, -\varphi) \sigma_1 \\
			&= \sigma_1 \mleft( \sigma_1 \sin{\theta} \cos{\varphi} 
			-
			\sigma_2 \sin{\theta} \sin{\varphi} 
			- \sigma_3 \cos{\theta}
			\mright) \matrixYkn{k}{-(n+1)} (-\theta, -\varphi) \sigma_1 \\
			&= - \sigma_1 \mleft( \sigma_1 \sin{(-\theta)} \cos{(-\varphi)} 
			+
			\sigma_2 \sin{(-\theta)} \sin{(-\varphi)} 
			+ \sigma_3 \cos{(-\theta)}
			\mright) \matrixYkn{k}{-(n+1)} (-\theta, -\varphi) \sigma_1 \\
			&= - \sigma_1 \mleft( \matrixYkn{k+1}{-(n+1)}(-\theta, -\varphi) \Akn{k}{-(n+1)}
			+
			\matrixYkn{k-1}{-(n+1)}(-\theta, -\varphi) \transposeAkn{k-1}{-(n+1)}
			\mright)  \sigma_1 \\
			&= - \matrixYkn{k+1}{n} (\theta, \varphi) \sigma_1
			\Akn{k}{-(n+1)} \sigma_1
			- \matrixYkn{k-1}{n} (\theta, \varphi) \sigma_1
			\transposeAkn{k-1}{-(n+1)} \sigma_1 .
		\end{align}
		Thus, by letting
		\begin{equation}
			\Akn{k}{n} \coloneqq - \sigma_1 \Akn{k}{-(n+1)} \sigma_1 , 
		\end{equation}
		we have the desired result for $n \leq -1$.
		
		Finally, we prove \eqref{eq:lemma 2}, \eqref{eq:lemma 3} and \eqref{eq:lemma 4}. 
		Since
		\begin{equation}
			\mleft( \sigma_1 \sin{\theta} \cos{\varphi} 
			+ 
			\sigma_2 \sin{\theta} \sin{\varphi} 
			+ \sigma_3 \cos{\theta}
			\mright)^2 = \matrixI, 
		\end{equation}
		we have
		\begin{equation}
			\mleft( \sigma_1 \sin{\theta} \cos{\varphi} 
			+ 
			\sigma_2 \sin{\theta} \sin{\varphi} 
			+ \sigma_3 \cos{\theta}
			\mright)^2 \matrixYkn{k}{n} = \matrixYkn{k}{n} .
		\end{equation}
		On the other hand, using \eqref{eq:lemma 1} twice, we also have
		\begin{align}
			&\quad \mleft( \sigma_1 \sin{\theta} \cos{\varphi} 
			+ 
			\sigma_2 \sin{\theta} \sin{\varphi} 
			+ \sigma_3 \cos{\theta}
			\mright)^2 \matrixYkn{k}{n} \\
			&= \mleft( \sigma_1 \sin{\theta} \cos{\varphi} 
			+ 
			\sigma_2 \sin{\theta} \sin{\varphi} 
			+ \sigma_3 \cos{\theta}
			\mright) ( \matrixYkn{k+1}{n} \Akn{k}{n}
			+
			\matrixYkn{k-1}{n} \transposeAkn{k-1}{n} ) \\
			&= \matrixYkn{k+2}{n} \Akn{k+1}{n} \Akn{k}{n} 
			+ \matrixYkn{k}{n} ( \transposeAkn{k}{n} \Akn{k}{n} + \Akn{k-1}{n} \transposeAkn{k-1}{n} )
			+ \matrixYkn{k-2}{n} \transposeAkn{k-2}{n} \transposeAkn{k-1}{n},
		\end{align}
		hence \eqref{eq:lemma 2} and \eqref{eq:lemma 3} hold.
		Moreover, since $\Akn{k}{n}, \Akn{k+1}{n} \ne \matrixO$ and $\Akn{k+1}{n} \Akn{k}{n} = \matrixO$, we obtain $\det\Akn{k}{n} = \det \Akn{k+1}{n} = 0$, which shows \eqref{eq:lemma 4}.
	\end{proof}
	As a consequence of \eqref{eq:lemma 2}, \eqref{eq:lemma 3} and \eqref{eq:lemma 4}, we obtain a certain orthonormal basis of $\C^2$:
	\begin{corollary} \label{cor:orthonormal basis for C^2}
		There exist $\{ \ukn{k}{n} \}, \{ \vkn{k}{n} \} \subset \C^2$ such that
		\begin{gather}
			\abs{ \ukn{k}{n} } = \abs{ \vkn{k}{n} } = 1 , \\
			\innerproduct{\ukn{k}{n}}{\vkn{k}{n}} = 0 , \\
			\Akn{k}{n} \ukn{k}{n} = \transposeAkn{k-1}{n} \vkn{k}{n} = \zerovec , \\
			\ukn{k+1}{n} = \Akn{k}{n} \vkn{k}{n} = \Akn{k}{n} \transposeAkn{k}{n} \ukn{k+1}{n}, \quad \vkn{k}{n} = \transposeAkn{k}{n} \ukn{k+1}{n} = \transposeAkn{k}{n} \Akn{k}{n} \vkn{k}{n} .
		\end{gather}
	\end{corollary} 
	\begin{proof}[Proof of Corollary \ref{cor:orthonormal basis for C^2}]
		Since $\det \Akn{k}{n} = 0$, we can take $\ukn{k}{n} \in \C^2$ satisfying $\Akn{k}{n} \ukn{k}{n} = \zerovec$ and $\abs{ \ukn{k}{n} } = 1$ for each $k$, $n$.
		Now let $\vkn{k}{n} \coloneqq \transposeAkn{k}{n} \ukn{k+1}{n}$.
		Then we have
		\begin{gather*}
			\Akn{k}{n} \vkn{k}{n} 
			= \Akn{k}{n} \transposeAkn{k}{n} \ukn{k+1}{n} 
			= ( \transposeAkn{k+1}{n} \Akn{k+1}{n} + \Akn{k}{n} \transposeAkn{k}{n} ) \ukn{k+1}{n} 
			= \ukn{k+1}{n} , \\
			\transposeAkn{k-1}{n} \vkn{k}{n} 
			= \transposeAkn{k-1}{n} \transposeAkn{k}{n} \ukn{k+1}{n} 
			= \zerovec , \\
			\innerproduct{\ukn{k}{n}}{\vkn{k}{n}}
			= \innerproduct{\ukn{k}{n}}{\transposeAkn{k}{n} \ukn{k+1}{n} } 
			= \innerproduct{\Akn{k}{n} \ukn{k}{n}}{ \ukn{k+1}{n} } \
			= 0 , \\
			\abs{ \vkn{k}{n} }^2 
			= \innerproduct{ \vkn{k}{n} }{ \transposeAkn{k}{n} \ukn{k+1}{n} } 
			= \innerproduct{ \Akn{k}{n} \vkn{k}{n} }{ \ukn{k+1}{n} } \
			= \abs{ u_{k+1}^{n} }^2 
			= 1 . \quad \qedhere
		\end{gather*}
	\end{proof}
	Now we prove the following lemma, which is analogue to Lemma \ref{lem:lemma 2D}:
	\begin{lemma} \label{lem:identity for 3D Ekn}
		Let 
		\begin{equation}
			\matrixEkn{k}{n}(\theta, \varphi) \coloneqq \begin{pmatrix}
				\matrixYkn{k}{n}(\theta, \varphi) \vkn{k}{n} & \zerovec & \zerovec & \matrixYkn{k+1}{n}(\theta, \varphi) \ukn{k+1}{n} \\
				\zerovec & \matrixYkn{k}{n}(\theta, \varphi) \vkn{k}{n} & \matrixYkn{k+1}{n}(\theta, \varphi) \ukn{k+1}{n} & \zerovec
			\end{pmatrix} .
		\end{equation}
		Then the following hold:
		\begin{eqenumerate}
			\eqitem \label{item:spherical harmonics decomposition 3D}
			For any $f \in L^2(\R^3, \C^4)$, there uniquely exists $\{ \fkn{k}{n} \} \subset L^2(\R_{>0}, \C^4)$ satisfying
			\begin{gather}
				f(\xi) =  r^{-1} \sum_{k = 0}^{\infty} \sumnk 
				\matrixEkn{k}{n}(\theta, \varphi) 
				\fkn{k}{n}(r) ,
				\quad \xi = (r \sin \theta \cos \varphi , r \sin \theta \sin \varphi, \cos \theta) , 
				\label{eq:spherical harmonics decomposition 3D} \tag{\ref{item:spherical harmonics decomposition 3D}.i} \\
				\norm{f}_{L^2(\R^3, \C^4)}^2 = \sum_{k = 0}^{\infty} \sumnk \norm{\fkn{k}{n}}_{L^2(\positiveR, \C^4)}^2 .
				\label{eq:spherical harmonics decomposition norm 3D} \tag{\ref{item:spherical harmonics decomposition 3D}.ii} 
			\end{gather}
			Conversely, for any $\{ \fkn{k}{n} \} \subset L^2(\R_{>0}, \C^4)$ satisfying
			\begin{equation}
				\sum_{k = 0}^{\infty} \sumnk \norm{\fkn{k}{n}}_{L^2(\positiveR, \C^4)}^2 < \infty ,
			\end{equation}
			the function $f$ given by \eqref{eq:spherical harmonics decomposition 3D} is in $L^2(\R^3, \C^4)$ and \eqref{eq:spherical harmonics decomposition norm 3D} holds.
			\eqitem \label{item:Sf norm decomposition Dirac 3D}
			Let $f \in L^2(\R^3, \C^4)$ and decompose $\fpm$ as
			\begin{equation}
				\fpm(\xi) = r^{-1} \sum_{k = 0}^{\infty} \sumnk 
				\matrixEkn{k}{n}(\theta, \varphi) 
				\fknpm{k}{n}(r) .
			\end{equation}
			Then we have
			\begin{align}
				&\quad \norm{\DiracS f}_{L^2(\R^{4}, \C^4)}^2 \\
				&= 2 \pi \sum_{k = 0}^{\infty} \sumnk \int_{0}^{\infty} 
				\mleft(
				\Innerproduct{%
					(\Lambdak{k}(r) \otimes \matrixI)
					\fknp{k}{n}(r)%
				}%
				{\fknp{k}{n}(r)} 
				+ \Innerproduct{%
					(\Lambdak{k}(r) \otimes \matrixI)
					\fknm{k}{n}(r)%
				}%
				{\fknm{k}{n}(r)} 
				\mright) \, dr ,
			\end{align}
			where
			\begin{equation}
				\Lambdak{k}(r) \coloneqq \begin{pmatrix}
					\lambda_k(r) & 0 \\
					0 & \lambda_{k+1}(r) 
				\end{pmatrix} .
			\end{equation}
			\eqitem \label{item:identity 3D} 
			We have
			\begin{align}
				\mleft( \alpha_1 \sin{\theta} \cos{\varphi} 
				+ 
				\alpha_2 \sin{\theta} \sin{\varphi} 
				+ \alpha_3 \cos{\theta}
				\mright) \matrixEkn{k}{n}(\theta, \varphi)
				&= \matrixEkn{k}{n}(\theta, \varphi) (\sigma_{1} \otimes \matrixI) , \\
				\beta \matrixEkn{k}{n}(\theta, \varphi) &= \matrixEkn{k}{n}(\theta, \varphi) (\sigma_{3} \otimes \sigma_{3} ).
			\end{align}
		\end{eqenumerate}
	\end{lemma}
	\begin{proof}[Proof of Lemma \ref{lem:identity for 3D Ekn}]
		\eqref{item:spherical harmonics decomposition 3D} and \eqref{item:Sf norm decomposition Dirac 3D} are immediate from Theorem \ref{thm:spherical harmonics decomposition} and Lemma \ref{lem:Sf norm decomposition Dirac}, respectively. 
		The second and third identities of \eqref{item:identity 3D} are also easy:
		\begin{align}
			&\quad \beta \matrixEkn{k}{n} (\theta, \varphi) \\
			&= \begin{pmatrix}
				\matrixI & \matrixO \\
				\matrixO & - \matrixI
			\end{pmatrix}
			\begin{pmatrix}
				\matrixYkn{k}{n} \vkn{k}{n} & \zerovec & \zerovec & \matrixYkn{k+1}{n} \ukn{k+1}{n}  \\
				\zerovec & \matrixYkn{k}{n} \vkn{k}{n} & \matrixYkn{k+1}{n} \ukn{k+1}{n}  & \zerovec
			\end{pmatrix} 
			\\
			&= 
			\begin{pmatrix}
				\matrixYkn{k}{n} \vkn{k}{n} & \zerovec  & \zerovec &  \matrixYkn{k+1}{n} \ukn{k+1}{n} \\
				\zerovec & -  \matrixYkn{k}{n} \vkn{k}{n} & - \matrixYkn{k+1}{n} \ukn{k+1}{n} & \zerovec  
			\end{pmatrix} \\
			&= 
			\begin{pmatrix}
				\matrixYkn{k}{n} \vkn{k}{n} & \zerovec & \zerovec & \matrixYkn{k+1}{n} \ukn{k+1}{n}  \\
				\zerovec & \matrixYkn{k}{n} \vkn{k}{n} & \matrixYkn{k+1}{n} \ukn{k+1}{n}  & \zerovec
			\end{pmatrix} 
			\begin{pmatrix}
				1 &0  & 0 & 0 \\
				0  & -1 & 0 & 0 \\
				0 & 0 & -1 &0 \\
				0 &0 & 0  & 1
			\end{pmatrix} \\
			&=  \matrixEkn{k}{n} (\theta, \varphi) (\sigma_{3} \otimes \sigma_{3} ) ,
		\end{align}
		\begin{align}
			&\quad ( \sigma_2 \otimes \matrixI ) \matrixEkn{k}{n} (\theta, \varphi) \\
			&= \begin{pmatrix}
				\matrixO & - i \matrixI \\
				i \matrixI & \matrixO
			\end{pmatrix}
			\begin{pmatrix}
				\matrixYkn{k}{n} \vkn{k}{n} & \zerovec & \zerovec & \matrixYkn{k+1}{n} \ukn{k+1}{n}  \\
				\zerovec & \matrixYkn{k}{n} \vkn{k}{n} & \matrixYkn{k+1}{n} \ukn{k+1}{n}  & \zerovec
			\end{pmatrix} 
			\\
			&= 
			\begin{pmatrix}
				\zerovec & - i \matrixYkn{k}{n} \vkn{k}{n} & - i \matrixYkn{k+1}{n} \ukn{k+1}{n} & \zerovec \\
				i \matrixYkn{k}{n} \vkn{k}{n} & \zerovec & \zerovec & i \matrixYkn{k+1}{n} \ukn{k+1}{n}
			\end{pmatrix} \\
			&= 
			\begin{pmatrix}
				\matrixYkn{k}{n} \vkn{k}{n} & \zerovec & \zerovec & \matrixYkn{k+1}{n} \ukn{k+1}{n}  \\
				\zerovec & \matrixYkn{k}{n} \vkn{k}{n} & \matrixYkn{k+1}{n} \ukn{k+1}{n}  & \zerovec
			\end{pmatrix} 
			\begin{pmatrix}
				0 & - i  & 0 & 0 \\
				i  & 0 & 0 & 0  \\
				0 & 0 & 0 &  i  \\
				0 &0 & - i  & 0
			\end{pmatrix} \\
			&=  \matrixEkn{k}{n} (\theta, \varphi) (\sigma_3 \otimes \sigma_2 ) .
		\end{align}
		The first identity of \eqref{item:identity 3D} follows from Lemma \ref{lem:identity for 3D} and Corollary \ref{cor:orthonormal basis for C^2}. Since
		\begin{gather}
			\begin{aligned}
				\mleft( \sigma_1 \sin{\theta} \cos{\varphi} 
				+ 
				\sigma_2 \sin{\theta} \sin{\varphi} 
				+ \sigma_3 \cos{\theta}
				\mright)
				\matrixYkn{k}{n} \ukn{k}{n} 
				\underset{\text{Lemma \ref{lem:identity for 3D}}}&{=} ( \matrixYkn{k+1}{n} \Akn{k}{n} + \matrixYkn{k-1}{n} \transposeAkn{k-1}{n} ) \ukn{k}{n} \\
				\underset{\text{Corollary \ref{cor:orthonormal basis for C^2}}}&{=} \matrixYkn{k-1}{n} \vkn{k-1}{n} , 	
			\end{aligned}\\
			\begin{aligned}
				\mleft( \sigma_1 \sin{\theta} \cos{\varphi} 
				+ 
				\sigma_2 \sin{\theta} \sin{\varphi} 
				+ \sigma_3 \cos{\theta}
				\mright)
				\matrixYkn{k}{n} \vkn{k}{n} 
				\underset{\text{Lemma \ref{lem:identity for 3D}}}&{=} ( \matrixYkn{k+1}{n} \Akn{k}{n} + \matrixYkn{k-1}{n} \transposeAkn{k-1}{n} ) \vkn{k}{n} \\
				\underset{\text{Corollary \ref{cor:orthonormal basis for C^2}}}&{=} \matrixYkn{k+1}{n} \ukn{k+1}{n} ,
			\end{aligned}
		\end{gather}
		we have
		\begin{align}
			&\quad \mleft( \alpha_1 \sin{\theta} \cos{\varphi} 
			+ 
			\alpha_2 \sin{\theta} \sin{\varphi} 
			+ \alpha_3 \cos{\theta}
			\mright) \matrixEkn{k}{n}(\theta, \varphi) \\
			&= ( \sigma_1 \otimes ( \sigma_1 \sin{\theta} \cos{\varphi} 
			+ 
			\sigma_2 \sin{\theta} \sin{\varphi} 
			+ \sigma_3 \cos{\theta} ) )
			\begin{pmatrix}
				\matrixYkn{k}{n} \vkn{k}{n} & \zerovec & \zerovec & \matrixYkn{k+1}{n} \ukn{k+1}{n} \\
				\zerovec & \matrixYkn{k}{n} \vkn{k}{n} & \matrixYkn{k+1}{n} \ukn{k+1}{n} & \zerovec
			\end{pmatrix} \\
			&= 
			\begin{pmatrix}
				\zerovec  & \matrixYkn{k+1}{n} \ukn{k+1}{n} & \matrixYkn{k}{n} \vkn{k}{n} & \zerovec  \\
				\matrixYkn{k+1}{n} \ukn{k+1}{n} & \zerovec & \zerovec & \matrixYkn{k}{n} \vkn{k}{n}
			\end{pmatrix} \\
			&= \matrixEkn{k}{n} (\sigma_{1} \otimes \matrixI). \quad \qedhere
		\end{align}
	\end{proof}
	Now we prove the following result:
	\begin{theorem} \label{thm:3D Dirac}
		We have
		\begin{equation}
			\norm{\DiracS}_{L^2(\R^3, \C^4) \to L^2(\R^4, \C^4)}^2 = 2 \pi \Diraclambdasup ,
		\end{equation}
		where
		\begin{gather}
			\Diraclambdasup
			\coloneqq \sup_{k \in \N} \sup_{r > 0} \Diraclambdak{k}(r), \\
			\Diraclambdak{k}(r) \coloneqq 
			\frac{1}{2} ( \lambda_k(r) + \lambda_{k+1}(r) ) + \frac{m}{2\phi_{m}(r)} \abs{ \lambda_k(r) - \lambda_{k+1}(r) } .
		\end{gather}
		Regarding extremisers, let $f \in L^2(\R^3, \C^4)$ be such that
		\begin{equation}
			f(\xi) = r^{-1} \sum_{k = 0}^{\infty} \sumnk 
			\matrixEkn{k}{n}(\theta, \varphi) 
			\fkn{k}{n}(r) .
		\end{equation}
		Then the following are equivalent:
		\begin{eqenumerate}
			\eqitem \label{item:extremiser 3D Dirac}
			The equality 
			\begin{equation}
				\norm{\DiracS f}_{L^2(\R^4, \C^4)}^2 = 2 \pi \Diraclambdasup \norm{f}_{L^2(\R^3, \C^4)}^2 
			\end{equation}
			holds.
			\eqitem \label{item:extremiser condition 3D Dirac}
			The functions $\{ \fkn{k}{n} \}$ satisfy
			\begin{gather}
				\supp{ \fkn{k}{n} } \subset \DiracL_k \coloneqq \set{ r > 0 }{ \Diraclambdak{k}(r) = \Diraclambdasup } , \\
				\fkn{k}{n}(r)
				\in W_k(r) , \quad \text{ a.e. } r > 0 
			\end{gather}
			for each $k \geq 0$ and $- k - 1 \leq n \leq k$, 
			where
			\begin{align}
				&W_k(r) \\
				&= 
				\begin{dcases}
					\C^4, & m ( \lambda_k(r) - \lambda_{k+1}(r) ) = 0 , \\
					\text{the eigenspace of $m \sigma_3 \otimes \matrixI + r \sigma_1 \otimes \sigma_3$ associated with $\phi_{m}(r)$}, 
					& m ( \lambda_k(r) - \lambda_{k+1}(r) )  > 0 , \\
					\text{the eigenspace of $m \sigma_3 \otimes \matrixI + r \sigma_1 \otimes \sigma_3$ associated with $- \phi_{m}(r)$},  
					& m ( \lambda_k(r) - \lambda_{k+1}(r) )  < 0 ,
				\end{dcases}
			\end{align}
		\end{eqenumerate}
		As a consequence, extremisers exist if and only if there exists 
		$k \in \N$ such that $\measure{\DiracL_k} > 0$.
	\end{theorem}
	
	\begin{proof}[Proof of Theorem \ref{thm:3D Dirac}]
		Let $f \in L^2(\R^3, \C^4)$ be such that 
		\begin{gather}
			f(\xi) =  r^{-1} \sum_{k = 0}^{\infty} \sumnk 
			\matrixEkn{k}{n}(\theta, \varphi)  \fkn{k}{n}(r) ,
			\quad \xi = (r \sin \theta \cos \varphi , r \sin \theta \sin \varphi, \cos \theta) ,  \\
			\norm{f}_{L^2(\R^3, \C^4)}^2 = \sum_{k = 0}^{\infty} \sumnk \norm{\fkn{k}{n}}_{L^2(\positiveR, \C^4)}^2  . 
		\end{gather}
		At first we need to compute $\fpm$.
		By \eqref{item:identity 3D}, we have
		\begin{align}
			\fpm
			&= 
			\frac{1}{2} \mleft( f \pm \frac{1}{\phi_{m}} \bigg( m \beta f + \sum_{j=1}^{3} \alpha_j \xi_j f \bigg) \mright) \\
			&= 
			\frac{1}{2r} \sum_{k=0}^{\infty} \sumnk 
			\matrixEkn{k}{n}(\theta, \varphi)
			\mleft(\matrixI_4 \pm \frac{1}{\phi_{m}} (m \sigma_3 \otimes \sigma_3 + r \sigma_1 \otimes \matrixI)\mright)
			\fkn{k}{n}(r) .
		\end{align}
		Therefore, \eqref{item:Sf norm decomposition Dirac 3D} implies
		\begin{equation}
			\norm{\DiracS f}_{L^2(\R^{4}, \C^4)}^2 \\ 
			= 2 \pi \sum_{k = 0}^{\infty} \sumnk \int_{0}^{\infty} 
			\innerproduct{
				\DiracLambdak{k}(r) \fkn{k}{n}(r)
			}
			{
				\fkn{k}{n}(r)
			}  \, dr ,
		\end{equation}
		where
		\begin{align}
			&\quad 2 \DiracLambdak{k}(r) \\
			&\coloneqq \frac{1}{2} (\matrixI_4 + \frac{1}{\phi_{m}} (m \sigma_3 \otimes \sigma_3 + r \sigma_1 \otimes \matrixI))
			(\Lambdak{k} \otimes \matrixI)
			(\matrixI_4 + \frac{1}{\phi_{m}} (m \sigma_3 \otimes \sigma_3 + r \sigma_1 \otimes \matrixI)) \\
			&\quad+ \frac{1}{2} (\matrixI_4 - \frac{1}{\phi_{m}} (m \sigma_3 \otimes \sigma_3 + r \sigma_1 \otimes \matrixI))
			(\Lambdak{k} \otimes \matrixI)
			(\matrixI_4 - \frac{1}{\phi_{m}} (m \sigma_3 \otimes \sigma_3 + r \sigma_1 \otimes \matrixI)) \\
			&= \Lambdak{k} \otimes \matrixI
			+ \frac{1}{\phi_{m}^2} (m \sigma_3 \otimes \sigma_3 + r \sigma_1 \otimes \matrixI) (\Lambdak{k} \otimes \matrixI) (m \sigma_3 \otimes \sigma_3 + r \sigma_1 \otimes \matrixI) \\
			&= \Lambdak{k} \otimes \matrixI
			+ \frac{1}{\phi_{m}^2} (m^2 \sigma_3 \Lambdak{k} \sigma_3 \otimes \matrixI + mr ( \sigma_3 \Lambdak{k} \sigma_1 +  \sigma_1 \Lambdak{k} \sigma_3 ) \otimes \sigma_3 + r^2 \sigma_1 \Lambdak{k} \sigma_1 \otimes \matrixI ) \\
			&= \Lambdak{k} \otimes \matrixI
			+ \frac{1}{\phi_{m}^2} (m^2 \Lambdak{k} \otimes \matrixI + mr ( \lambda_k - \lambda_{k+1} ) \sigma_{1} \otimes \sigma_{3}  + (\phi_{m}^2 - m^2) \sigma_1 \Lambdak{k} \sigma_1 \otimes \matrixI ) \\
			&=  (\Lambdak{k} + \sigma_1 \Lambdak{k} \sigma_1) \otimes \matrixI
			+ \frac{m}{\phi_{m}^2} (m ( \Lambdak{k} - \sigma_1 \Lambdak{k} \sigma_1 ) \otimes \matrixI + r ( \lambda_k - \lambda_{k+1} ) \sigma_{1} \otimes \sigma_{3} ) ) \\
			&= (\lambda_k + \lambda_{k+1}) \matrixI_4
			+ \frac{m}{\phi_{m}^2} (\lambda_k - \lambda_{k+1}) (m \sigma_3 \otimes \matrixI + r \sigma_1 \otimes \sigma_3 ) .
		\end{align}
		Now we need to determine the maximal eigenvalue of $\DiracLambdak{k}(r)$ and its associated eigenspace. 
		Since eigenvalues of the matrix $m \sigma_3 \otimes \matrixI + r \sigma_1 \otimes \sigma_3$ are $\pm \phi_{m}(r)$, 
		we conclude that the maximal eigenvalue of $\DiracLambdak{k}(r)$ and its associated eigenspace are
		\begin{equation}
			\Diraclambdak{k}(r) 
			= \frac{1}{2} ( \lambda_k(r) + \lambda_{k+1}(r) ) + \frac{m}{2\phi_{m}(r)} \abs{ \lambda_k(r) - \lambda_{k+1}(r) }
		\end{equation}
		and
		\begin{align}
			&W_k(r) \\ 
			&= 
			\begin{dcases}
				\C^4, & m ( \lambda_k(r) - \lambda_{k+1}(r) ) = 0 , \\
				\text{the eigenspace of $m \sigma_3 \otimes \matrixI + r \sigma_1 \otimes \sigma_3$ associated with $\phi_{m}(r)$}, 
				& m ( \lambda_k(r) - \lambda_{k+1}(r) )  > 0 , \\
				\text{the eigenspace of $m \sigma_3 \otimes \matrixI + r \sigma_1 \otimes \sigma_3$ associated with $- \phi_{m}(r)$},  
				& m ( \lambda_k(r) - \lambda_{k+1}(r) )  < 0 ,
			\end{dcases}
		\end{align}
		respectively.
		Therefore, we have
		\begin{align}
			\norm{\DiracS f}_{L^2(\R^{4}, \C^4)}^2 
			&= 2 \pi \sum_{k = 0}^{\infty} \sumnk \int_{0}^{\infty} 
			\innerproduct{
				\DiracLambdak{k}(r) \fkn{k}{n}(r)
			}
			{
				\fkn{k}{n}(r)
			}  \, dr \\
			&\leq 2 \pi \sum_{k = 0}^{\infty} \sumnk \int_{0}^{\infty}  \Diraclambdak{k}(r) \abs{\fkn{k}{n}(r)}^2 \, dr \\
			&\leq 2 \pi \Diraclambdasup \norm{f}_{L^2(\R^3, \C^4)}^2 
		\end{align}
		and hence
		\begin{equation}
			\norm{\DiracS}^2 \leq 2 \pi \Diraclambdasup .
		\end{equation}
		
		The equality $\norm{\DiracS}^2 = 2 \pi \Diraclambdasup$ and the characterization of extremisers \eqref{item:extremiser 3D Dirac}$\iff$\eqref{item:extremiser condition 3D Dirac} can be proved by the same argument as in the proof of Theorem \ref{thm:2D Dirac}, so we omit the details.
	\end{proof}
	\section{Explicit values} \label{section:explicit value}
	In this section, we prove Theorems \ref{thm:type B Dirac} and \ref{thm:type C Dirac}.
	In order to prove Theorem \ref{thm:type B Dirac}, we use the following lemma:
	\begin{quotelemma}[\cite{BS2017}] \label{lem:BS 1.6}
		Let $d \geq 2$, $1 < s < d$, $a > 0$, and suppose that $( w, \psi, \phi )$ satisfy
		\begin{equation}
			w(r) = r^{-s}, \quad \psi(r)^2 = a r^{1-s} \abs{ \phi'(r) } .
		\end{equation}
		Then we have
		\begin{equation}
			\frac{1}{(2\pi)^{d-1}} \lambda_k(r) = a c_k \coloneqq 2^{2-s} a \pi \frac{ \Gamma(s-1) \Gamma((d-s)/2 + k) }{ ( \Gamma(s/2) )^2 \Gamma( (d + s)/2 + k - 1 ) } .
		\end{equation}
		Furthermore, $\{ c_k \}_{k \in \N}$ is strictly decreasing.
		For example, if $d \geq 3$ and $s = 2$, then
		\begin{equation}
			\frac{1}{(2\pi)^{d-1}} \lambda_k(r) = a c_k = a \frac{2 \pi }{d + 2k - 2} .
		\end{equation}
	\end{quotelemma}
	The proof of Lemma \ref{lem:BS 1.6} can be found in \cite[Proof of Theorem 1.6]{BS2017}.
	Note that \cite[Theorem 1.6]{BS2017} is immediate from Theorem \ref{thm:Schrodinger} and Lemma \ref{lem:BS 1.6}. 
	In fact, we have
	\begin{align}
		C_{d}(r^{-s},r^{(2-s)/2},r^2) 
		&= \frac{1}{(2\pi)^{d-1}}\sup_{k\in\N}\sup_{r>0}\lambda_{k}(r) \\
		&=  c_0 / 2 \\
		&= 2^{1-s} \pi \frac{ \Gamma(s-1) \Gamma((d-s)/2) }{ ( \Gamma(s/2) )^2 \Gamma( (d + s)/2 - 1 ) } .
	\end{align}
	\begin{proof}[Proof of Theorem \ref{thm:type B Dirac}]
		At first, we prove the case $d=2, 3$.
		In this case,
		Lemma \ref{lem:BS 1.6} implies
		\begin{equation}
			\Diraclambda_{k}(r) = \frac{1}{2} \mleft( c_k + c_{k+1} + \frac{m}{\sqrt{r^2+m^2}} ( c_{k} - c_{k+1} ) \mright) ,
		\end{equation}
		and so that 
		\begin{gather}
			\begin{aligned}
				\sup_{k \in \N} \sup_{r > 0} \Diraclambda_{k}(r)
				&= \sup_{k \in \N} \lim_{r \to +0} \Diraclambda_{k}(r) \\
				&=
				\begin{cases}
					\sup_{k \in \N} \frac{1}{2} \mleft( c_k + c_{k+1}  \mright) , & m = 0 , \\
					\sup_{k \in \N} c_k , & m > 0 
				\end{cases} \\
				&= \begin{cases}
					(c_0 + c_1)/2 , & m = 0 , \\
					c_0 , & m > 0 ,
				\end{cases} 
			\end{aligned} \\
			\set{ r > 0 }{ \Diraclambdak{k}(r) = \Diraclambdasup }
			= 
			\begin{cases}
				\positiveR, & m = k = 0, \\
				\emptyset, & \text{else.} 
			\end{cases} 
		\end{gather}
		Combining these with Theorems \ref{thm:2D Dirac} and \ref{thm:3D Dirac}, we obtain the desired result in the case $d = 2, 3$.
		
		In the case $d \geq 4$ with $m > 0$, notice that we have
		\begin{equation}
			\frac{1}{ (2\pi)^{d-1} } \sup_{r>0} \Diraclambda_\rad(r) 
			\underset{\text{Thm.~\ref{thm:radial Dirac}}}{\leq} 
			\widetilde{C}_{d}(r^{-s},\phi_{m}(r)^{-1/2} r^{(2-s)/2},m) 
			\underset{\text{Cor.~\ref{cor:Dirac Schrodinger equivalence}}}{\leq}
			2 C_{d}(r^{-s},r^{(2-s)/2},r^2) 
			\underset{\text{Thm.~\ref{thm:Schrodinger explicit value}}}{=} c_0 .
		\end{equation}
		On the other hand, Lemma \ref{lem:BS 1.6} implies
		\begin{equation}
			\frac{1}{ (2\pi)^{d-1} }  \Diraclambda_\rad(r) = \frac{1}{2}\mleft( c_0 + c_1 + \frac{m^2}{r^2 + m^2}( c_0 - c_1 ) \mright) ,
		\end{equation}
		and so that
		\begin{equation}
			\frac{1}{ (2\pi)^{d-1} } \sup_{r>0} \Diraclambda_\rad(r) = c_0 . 
		\end{equation}
		Therefore, we conclude that 
		\begin{equation}
			\widetilde{C}_{d}(r^{-s},\phi_{m}(r)^{-1/2} r^{(2-s)/2},m)  =  c_0
		\end{equation}
		holds when $d \geq 4$ with $m > 0$. 
	\end{proof}
	Finally, we prove Theorem \ref{thm:type C Dirac}.
	\begin{proof}[Proof of Theorem \ref{thm:type C Dirac}]
		Combining Corollary \ref{cor:Dirac Schrodinger equivalence} and Theorem \ref{thm:Schrodinger explicit value}, we immediately obtain
		\begin{align}
			\widetilde{C}_{d}((1+r^2)^{-s/2},\phi_{m}(r)^{-1/2} r^{1/2},m) 
			\underset{\text{Cor.~\ref{cor:Dirac Schrodinger equivalence}}}&{\leq}
			2 C_{d}((1+r^2)^{-s/2}, r^{1/2} ,r^2)  \\
			\underset{\text{Thm.~\ref{thm:Schrodinger explicit value}}}&{=} \pi^{1/2} \frac{ \Gamma( (s-1)/2 ) }{ \Gamma(s) } .
		\end{align}
		On the other hand, we have
		\begin{align}
			\widetilde{C}_{d}((1+r^2)^{-s/2},\phi_{m}(r)^{-1/2} r^{1/2},m)
			\underset{\text{Thm.~\ref{thm:radial Dirac}}}&{\geq} \frac{1}{ (2\pi)^{d-1} } \sup_{r>0} \Diraclambda_k(r) \\
			&\geq \frac{1}{2 (2\pi)^{d-1}} \limsup_{r \to \infty}\mleft( \lambda_0(r) + \lambda_1(r) + \frac{m^2}{r^2 + m^2}( \lambda_0(r) - \lambda_1(r) ) \mright) ,
		\end{align}
		where $\lambda_k$ is that associated with $((1+r^2)^{-s/2},\phi_{m}(r)^{-1/2} r^{1/2},\phi_m)$.
		Hence, it is enough to show that
		\begin{equation}
			\frac{1}{2 (2\pi)^{d-1}} \limsup_{r \to \infty}\mleft( \lambda_0(r) + \lambda_1(r) + \frac{m^2}{r^2 + m^2}( \lambda_0(r) - \lambda_1(r) ) \mright) = \pi^{1/2} \frac{ \Gamma( (s-1)/2 ) }{ \Gamma(s) } .
		\end{equation}
		This follows from the fact that
		\begin{equation}
			\frac{1}{ (2\pi)^{d-1} } \lim_{r \to \infty} \lambda_k(r) = \pi^{1/2} \frac{ \Gamma( (s-1)/2 ) }{ \Gamma(s) } 
		\end{equation}
		holds for each $k \in \N$.
		See \cite[Theorem 2.2]{BS2017} and \cite[(1.11)]{BSS2015} for details.
	\end{proof}
	
	\section*{Acknowledgments}
	The authors would like to thank Mitsuru Sugimoto (Nagoya University) and Yoshihiro Sawano (Chuo University) for valuable discussions.
	\setcitestyle{numbers} 

\end{document}